\documentclass[12pt]{amsart}
\usepackage{amssymb, amsmath, amsfonts}

\def\phi{\varphi}

\def\a{\alpha}
\def\b{\beta}
\def\g{\gamma}
\def\G{\Gamma}

\def\l{\lambda}
\def\L{\Lambda}

\def\1#1{\overline{#1}}
\def\2#1{\widetilde{#1}}
\def\3#1{\widehat{#1}}
\def\4#1{\mathbb{#1}}
\def\5#1{\mathfrak{#1}}
\def\6#1{{\mathcal{#1}}}

\def\T{{\Theta}}
\newtheorem{theorem}{Theorem}[section]

\newtheorem{corollary}[theorem]{Corollary}

\newtheorem{definition}[theorem]{Definition}
\newtheorem{lemma}[theorem]{Lemma}

\newtheorem{proposition}[theorem]{Proposition}

\theoremstyle{remark}
\newtheorem{example}[theorem]{Example}

\begin{document}
\numberwithin{equation}{section}

\def\bl{\begin{Lem}}
\def\el{\end{Lem}}
\def\bp{\begin{Pro}}
\def\ep{\end{Pro}}
\def\bt{\begin{Thm}}
\def\et{\end{Thm}}
\def\bc{\begin{Cor}}
\def\ec{\end{Cor}}
\def\bd{\begin{Def}}
\def\ed{\end{Def}}
\def\br{\begin{Rem}}
\def\er{\end{Rem}}
\def\be{\begin{example}}
\def\ee{\end{example}}
\def\bpf{\begin{proof}}
\def\epf{\end{proof}}
\def\ben{\begin{enumerate}}
\def\een{\end{enumerate}}
\def\beq{\begin{equation}}
\def\eeq{\end{equation}}

\title[Holomorphic maps ]{Holomorphic maps between closed $SU(\ell,m)$-orbits in Grassmannian}

\author[S.-Y. Kim ]{Sung-Yeon Kim}
\address{Korea Institute for Advanced Study, 85 Hoegiro, Dongdaemun-gu,
Seoul, 02455,
Korea}
\email{sykim8787@kias.re.kr}

\thanks{This research was supported by Basic Science Research Program through the National Research Foundation of Korea(NRF) funded by the Ministry of  Science, ICT and Future Planning(grant number NRF-2015R1A2A2A11001367)}

\keywords{homogeneous CR manifold, CR embedding, totally geodesic embedding, minimal orbit of
a real form}
\subjclass[2010]{
32V40, 32V30, 32V05,
32M10, 14M15}

\def\Label#1{\label{#1}}

\begin{abstract}
Orbits of $SU(\ell, m)$ in a Grassmannian manifold have homogeneous CR structures.
In this paper, we study germs of smooth CR mappings sending a closed orbit of $SU(\ell,m)$
into a closed orbit of $SU(\ell',m')$ in Grassmannian manifolds. We show that if the signature difference
of the Levi forms of two orbits
is not too large, then the mapping can be factored into a simple form and one of the factors extends to a totally geodesic embedding of the ambient Grassmannian into another Grassmannian with respect to the standard metric.
As an application, we give a sufficient condition for a smooth CR mapping sending a closed orbit of $SU(\ell,m)$
into a closed orbit of $SU(\ell',m')$ in Grassmannian manifolds to extend as a totally geodesic embedding of the Grassmannian into another Grassmannian.
\end{abstract}

\maketitle

\section{Introduction}
Rigidity of holomorphic maps between open pieces of a sphere was first studied by Poincar\'{e} \cite{P07}, who proved the rigidity of holomorphic maps sending open piece of sphere into another in dimension 2 and later by Alexander \cite{A74} in arbitrary dimension. This result was generalized
for holomorphic maps between pieces of spheres of different
dimension by Webster \cite{W79}, Faran \cite{Fa86}, Cima-Suffridge \cite{CS83, CS90}, Forstneric \cite{F86, F89}, Huang \cite{H99, H03}, Huang-Ji \cite{HJ01} and Huang-Ji-Xu \cite{HJX06}.

Ball is a bounded symmetric domain of rank one. Rigidity of proper holomorphic maps between general bounded symmetric domains was conjectured by Mok \cite{M89} and proved by Tsai \cite{T93}, showing that they are necessarily totally geodesic with respect to the Bergmann metric if the rank of the source is greater or equal to that of the target. In relation with it, rigidity of holomorphic maps between open pieces of the boundary orbits of bounded symmetric domains was proved by Henkin-Tumanov \cite{H-T} for automorphisms and by Kim-Zaitsev \cite{KZ13,KZ14} for the type I bounded symmetric domains of different rank. We refer the reader
to the work  \cite{CaMo}, \cite{M89, M08, M11} and the references
therein for
various related rigidity problems for holomorphic maps between complex hyperbolic
space forms and Hermitian symmetric spaces.

For CR manifolds with mixed Levi signature, rigidity phonemena for CR
maps between real hyperquadrics and proper holomorphic maps between generalised balls have been studied by Chern-Moser \cite{CM74}, Ebenfelt-Huang-Zaitsev \cite{EHZ04, EHZ05}, Baouendi-Huang \cite{BH05}, Baouendi-Ebenfelt-Huang \cite{BEH09}.
These results were generalised by Ng \cite{Ng}, who studied the rigidity of holomorphic maps between minimal $SU(\ell,m)$-orbits in Grassmannians and proper holomorphic maps between corresponding flag domains.

Let $p,q,\ell$ be positive integers such that $q\leq \ell\leq p$. Denote by $Gr(q,p)$ the Grassmannian of $q$-planes in $\mathbb{C}^{p+q}$ and by
$S_{q,p}^\ell$ the minimal $SU(\ell,m)$-orbit in $Gr(q, p),$ where $m=p+q-\ell$. In \cite{Ng}, Ng showed that maximal complex manifolds in $S_{q,p}^\ell$ are totally geodesic subgrassmannians $Gr(q, \ell)$ and hence they can be parameterized by $S_{\ell,m}^\ell$, which is the Shilov boundary of a type I bounded symmetric domain in $ Gr(\ell, m)$. More generally, $n$-confined subgrassmannians(see \S3 for definition) for $q\leq n\leq \ell$ can be parameterized by $S_{n, p+q-n}^{\ell}$. For each $n$, one can define the universal space of $n$-confined subgrassmannians over $S_{q,p}^\ell.$ These are homogeneous CR manifolds in flag manifolds and play an important roll in the study of CR maps so do the characteristic bundles in the study of proper holomorphic maps between bounded symmetric domains. See \cite{M89} for characteristic bundles and related topics.

Under the condition in \cite{Ng}, maximal complex submanifolds are the same $Gr(q, \ell)$ for source and target orbits which enables one to lift the given holomorphic map as meromorphic maps between the universal spaces of maximal complex submanifolds over the orbits. This lifting shows that the given holomorphic map preserves the characeristic bundles over the maximal complex submanifolds. Then Ng used the result of Mok \cite{M08} to obtain the rigidity of the holomorphic map between minimal orbits.

The goal of this paper is to generalize the results of Ng in the case when maximal complex submanifolds of the source and target orbits are different. First we investigate the CR structures of $S_{q,p}^\ell$ and the universal space of $n$-confined subgrassmannians over it. If $q$ is strictly less than $\ell$, every two points in $S_{q,p}^\ell$ are connected by chains of maximal complex submanifolds. We shows that the rigidity phenomenon propagates along chains of maximal complex submanifolds if the signature difference of the Levi forms of source and target orbits is small.  More precisely, we prove the following.

\begin{theorem}\label{main}
Let $f:(S_{q,p}^\ell,Z)\to (S_{q',p'}^{\ell'},Z')$ be a germ of a smooth transversal CR embedding (See Definition~\ref{cr-trans} for transversality).
Assume that $q>1$ and
\begin{equation}\Label{main-ineq}
\ell'-q'<2(\ell-q).
\end{equation}
Then it follows that $q'\geq q,~p'\geq p$
and after composing with suitable automorphisms of $S_{q,p}^{\ell}$ and $S_{q',p'}^{\ell'}$, that $f$ is of the form
\begin{equation*}
z
\mapsto
(f_1(z), f_2(z))\in Gr(q',L)\times Gr(q', N),
\end{equation*}
where
$f_1$ is a standard embedding of $Gr(q, p)$ into $Gr(q', L)$ for some subspace $L\subset \mathbb{C}^{p'+q'}$
and $f_2$ is a holomorphic map from $S_{q,p}^{\ell}$ into $Gr(q', N)$
contained in $S_{q',p'}^{\ell'}$ with $\dim N\leq \ell'-\ell+q.$
\end{theorem}
As a corollary, we obtain the following.
\begin{corollary}\label{main-cor}
Let $f:S_{q,p}^\ell \to S_{q',p'}^{\ell'}$ be a (globally defined) smooth transversal CR embedding.
Assume that $q>1$ and $\ell'-q'<2(\ell-q)$. Assume further that
\begin{equation}\Label{main-ineq}
q'\times(\ell'-\ell+q-q')<q\times (\ell-q).
\end{equation}
Then $f$ extends to $Gr(q, p)$ as a standard embedding into $Gr(q', p')$.
\end{corollary}

We remark that Theorem~\ref{main} is proved in \cite{Ng} for $q=q'$ and $\ell=\ell'$. In this case $f_2$ is a constant map and therefore $f$ extends
to a standard embedding of $Gr(q,p)$ into $Gr(q, p').$

\medskip

\medskip

We use the method of moving frames. \S1 and \S2 are devoted to investigate CR structure of $S_{q,p}^\ell$ and its and maximal complex submanifolds.
In \S3, we construct moving frames adapted to $f$. Then in \S4, we reduce the freedom of adapted frames using properties of universal space of maximal complex submanifolds over $S_{q,p}^\ell$, which enables us to apply the result of Mok \cite{M08}.
In the last section, we prove Theorem~\ref{main} and Corollary~\ref{main-cor}.


\section{Preliminaries and adapted frames}\Label{preli}

Let $p,q$ be positive integers such that $q< p$.
For an integer $\ell$ such that $q\leq \ell\leq (p+q)/2$,
define a Hermitian inner product $\langle~,\rangle_{\ell,m}$ in $\mathbb{C}^{p+q}$ by
\begin{equation}\label{basic inner}
\langle u,v\rangle_{\ell,m}:=-(u_1\bar v_1+\cdots+u_\ell\bar v_\ell)+(u_{\ell+1}\bar
v_{\ell+1}+\cdots+u_{p+q}\bar v_{p+q}),
\end{equation}
where $m=p+q-\ell$,
$u=(u_1,\ldots,u_{p+q})$ and $v=(v_1,\ldots,v_{p+q})$.
Define
\begin{equation*}
{S}^{\ell}_{q,p}:=\{Z\in Gr(q,p):\langle ~, \rangle_{\ell,m} |_Z=0 \}.
\end{equation*}
Then $S^{\ell}_{q,p}$ is the unique closed $SU(\ell, m)$-orbit in $Gr(q,p)$.
Note that if $q=\ell$, then $S_{q,p}^q$ is the Shilov boundary of the bounded symmetric domain
$D_{q,p}$ defined by
\begin{equation*}
{D}_{q,p}:=\{Z\in \mathbb{C}^{q\times p}:I_q-Z\overline Z^t>0 \},
\end{equation*}
where $I_q$ is the $q\times q$ identity matrix.

A \emph{Grassmannian frame adapted to} $S^{\ell}_{q,p}$, or simply
$S^{\ell}_{q,p}$-\emph{frame} is a frame $\{ Z_1,\ldots,Z_{p+q}\}$ of
$\mathbb{C}^{p+q}$ with $\det(Z_1,\ldots,Z_{p+q})=1$ such that
\begin{equation}\Label{structure}
\langle Z_\alpha,Z_{p+\beta}\rangle_{\ell,m}=\langle
Z_{p+\a},Z_\b\rangle_{\ell,m}=~\delta_{\alpha\beta},~
 \langle Z_{q+j},Z_{q+k}\rangle_{\ell,m} =\3\delta_{jk},~ \a,\b=1,\ldots,q,~j,k=1,\ldots,p-q
\end{equation}
and
\begin{equation*}\label{structure2}
\langle Z_\L,Z_\G\rangle_{\ell,m}=0~\text{   otherwise,   }
\end{equation*}
where $\3\delta_{jk}=-\delta_{jk}$ if $\min(j,k)\leq \ell-q$,
$\3\delta_{jk}=\delta_{jk}$ otherwise, and the capital Greek indices $\L,\G,\Omega$ etc.\ run
from $1$ to $p+q$, i.e. the scalar product $\langle\cdot,\cdot\rangle_{\ell,m}$
in basis $\{Z_1,\ldots,Z_{p+q}\}$ is given by the matrix
$$
\begin{pmatrix}
0&0&0& I_{q}\\
0&-I_{\ell-q}&0&0\\
0&0&I_{m-q}&0\\
I_{q}&0&0&0\\
\end{pmatrix}.
$$
 We also use the notation
\begin{eqnarray*}
Z:&=&(Z_1,\ldots,Z_q),\\
X=(X_1,\ldots,X_{p-q}):&=&(Z_{q+1},\ldots,Z_{p}),\\
Y=(Y_1,\ldots,Y_q):&=&(Z_{p+1},\ldots,Z_{p+q})
\end{eqnarray*}
so that \eqref{structure} can be rewritten as
\begin{equation*}\Label{structure'}
\langle Z_{\a}, Y_{\b}\rangle_{\ell,m} = \langle Y_{\b}, Z_{\a}\rangle_{\ell,m}= \delta_{\a\b},
\quad \langle X_{j}, X_{k}\rangle_{\ell,m} =\3\delta_{jk}.
\end{equation*}

Let $\mathcal{B}^{\ell}_{q,p}$ be the set of all
$S^{\ell}_{q,p}$-frames. Then $\mathcal{ B}^{\ell}_{q,p}$ can be identified with
$SU(\ell,m)$ by the left action. By abuse of notation, we also denote by $Z$
the $q$-dimensional subspace
of $\mathbb{C}^{p+q}$ spanned by $Z_1,\ldots,Z_q$.
Then we can regard $\mathcal{B}^{\ell}_{q,p}$ as a bundle over $S^{\ell}_{q,p}$
with respect to a natural projection $(Z, X, Y)\to Z$. The Maurer-Cartan form
$\pi=(\pi_\L^{~\G})$ on $\mathcal{B}^{\ell}_{q,p}$ is an $su(\ell, m)$-valued one form given by the equation
\begin{equation}\Label{differential}
dZ_\L=\pi_\L^{~\G}Z_\G
\end{equation}
satisfying the structure equation
\begin{equation*}\Label{struc-eq}
d\pi_\L^{~\G}=\pi_\G^{~\Omega}\wedge\pi_\Omega^{~\G}.
\end{equation*}
We use the block matrix representation
with respect to the basis $(Z,X,Y)$ to write
\begin{equation*}\Label{pi}
\begin{pmatrix}
\pi_{\a}^{~\b} & \pi_{\a}^{~q+j} & \pi_{\a}^{~p+\b}\\
\pi_{q+k}^{~\b} & \pi_{q+k}^{~q+j}  & \pi_{q+k}^{~p+\b}\\
\pi_{p+\a}^{~\b} & \pi_{p+\a}^{~q+j}  & \pi_{p+\a}^{~p+\b}\\
\end{pmatrix}
= :
\begin{pmatrix}
\psi_{\a}^{~\b} & \theta_{\a}^{~j} & \phi_{\a}^{~\b}\\
\sigma_{k}^{~\b} & \omega_{k}^{~j} & \theta_{k}^{~\b}\\
\xi_{\a}^{~\b} & \sigma_{\a}^{~j}   &\3\psi_{\a}^{~\b}\\
\end{pmatrix},
\end{equation*}
which satisfies the symmetry relations
\begin{equation*}\Label{symmetries}
\begin{pmatrix}
\psi_{\a}^{~\b} & \theta_{\a}^{~j} & \phi_{\a}^{~\b}\\
\sigma_{k}^{~\b} & \omega_{k}^{~j} & \theta_{k}^{~\b}\\
\xi_{\a}^{~\b} & \sigma_{\a}^{~j} & \3\psi_{\a}^{~\b}\\
\end{pmatrix}
=-
\begin{pmatrix}
\3\psi_{\bar\b}^{~\bar\a} & \3\delta_{j}^i\theta_{\bar i}^{~\bar\a}  & \phi_{\bar\b}^{~\bar\a}\\
\3\delta_{i}^k\sigma_{\bar\b}^{~\bar i} &\3\delta_{i}^k \omega_{\bar j}^{~\bar i} & \3\delta_{i}^k\theta_{\bar\b}^{~\bar i}\\
\xi_{\bar\b}^{~\bar\a} &\3\delta_{j}^i \sigma_{\bar i}^{~\bar\a} & \psi_{\bar\b}^{~\bar\a}\\
\end{pmatrix}
\end{equation*}
that follow directly by differentiating \eqref{structure}.

The defining equations
of $S^{\ell}_{q,p}$ can be written as
$$\langle Z_\a, Z_\b\rangle_{\ell,m}=0,\quad \a,\b=1,\ldots,q$$
and hence their differentiation yields
\begin{equation}\label{diff-eqn}
\langle dZ_\a ,Z_\b\rangle_{\ell,m}
+\langle
Z_\a,dZ_\b\rangle_{\ell,m}=0.
\end{equation}
By substituting $dZ_\Lambda=\pi_\Lambda^{~\Gamma}Z_\Gamma$ into $(1,0)$ component of \eqref{diff-eqn} we obtain
\begin{equation*}
\phi_\alpha^{~\gamma}\langle Y_\gamma,Z_\beta\rangle_{\ell,m}=
\phi_\alpha^{~\beta}=0,
\end{equation*}
when restricted to the $(1,0)$ tangent space.
Comparing the dimensions, we conclude that the kernel of
$\{\phi_\alpha^{~\beta},\alpha,\beta=1,\ldots,q\}$ forms the CR
bundle of $S^{\ell}_{q,p}$, i.e.
$$\ker(\phi|_Z)=T^{1,0}_Z S^{\ell}_{q,p}\oplus T^{0,1}_Z  S^{\ell}_{q,p}.$$
In other words, $\phi=(\phi_\alpha^{~\beta})$ spans the space of contact
forms on $S^{\ell}_{q,p}$.
Since $$dZ_\alpha=\psi_\alpha^{~\beta}Z_\beta+\theta_\alpha^{~ j}X_j+\phi_\alpha^{~\beta}Y_\beta$$
and $\phi=(\phi_\alpha^{~\beta})$ is a contact form
at $Z=(Z_1,\ldots,Z_q)$, we conlcude that $\phi_{\a}^{~\b}$
and $\theta_{\a}^{~j}$ together form a basis in the space of
all $(1,0)$ forms of $S^{\ell}_{q,p}$.
The Levi form is given by
$$d \phi_\a^{~\b}=\theta_\a^{~j}\wedge\theta_j^{~\b}=-\sum_{j=1}^{\ell-q}\theta_\a^{~j}\wedge\overline{\theta_\b^{~j}}
+\sum_{j=\ell-q+1}^{p-q} \theta_\a^{~j}\wedge\overline{\theta_\b^{~j}}\mod\phi.$$
Therefore if $\ell>q,$ then the image of the Levi map becomes the complex normal bundle of $S_{q,p}^\ell.$

\medskip

For a change of frame given by
$$
\begin{pmatrix}
\2Z\\
\2X\\
\2Y
\end{pmatrix}
:=U
\begin{pmatrix}
Z\\
X\\
Y
\end{pmatrix}
,$$
$\pi$ changes via
$$\widetilde \pi=dU\cdot U^{-1}+U\cdot\pi\cdot U^{-1}.$$
There are several types of frame changes.

\begin{definition}\Label{changes}
{\rm We call a change of frame}
\begin{enumerate}
\item[i)]change of position {\rm if}
$$
\widetilde Z_\alpha=W_\alpha^{~\beta}Z_\beta,\quad
\widetilde Y_\alpha=V_\alpha^{~\beta}Y_\beta,\quad
\widetilde X_j=X_j,
$$
{\rm where $W=(W_\alpha^{~\beta})$ and $V=(V_\alpha^{~\beta})$ are
$q\times q$ matrices satisfying $\overline{V^t}W=I_q$};

\item[ii)]change of real vectors {\rm if}
$$
\widetilde Z_\alpha=Z_\alpha,\quad
\widetilde X_j=X_j,\quad
\widetilde Y_\alpha=Y_\alpha+H_\alpha^{~\beta}Z_\beta,
$$
{\rm where $H=(H_\alpha^{~\beta})$ is a hermitian matrix};

\item[iii)]dilation {\rm if}
$$
\widetilde Z_\alpha=\lambda_{\alpha}^{-1}Z_\alpha,\quad
\widetilde Y_\alpha=\lambda_\alpha Y_\alpha,\quad
\widetilde X_j=X_j,
$$
{\rm where $\lambda_\alpha>0$};

\item[iv)]rotation {\rm if}
$$
\widetilde Z_\alpha=Z_\alpha,\quad
\widetilde Y_\alpha=Y_\alpha,\quad
\widetilde X_j=U_j^{~k}X_k,
$$
{\rm where $(U_j^{~k})$ is an $SU(\ell-q,m-q)$ matrix.}
\end{enumerate}
\end{definition}

Change of position in Definition~\ref{changes} sends $\phi$ and $\theta$ to
$$
\widetilde
\phi_\alpha^{~\beta}=W_\alpha^{~\gamma}\phi_\gamma^{~\delta}W^{*}{}_{\delta}^{~\b},
\quad W^{*}{}_{\delta}^{~\b}=\overline{W_{\beta}^{~\delta}},\quad
\widetilde\theta_\alpha^{~j}=W_\alpha^{~\beta}\theta_\beta^{~j}.
$$
Dilation changes $\phi_\a^{~\b}$, $\theta_\a^{~j}$ to
$$
\widetilde
\phi_\alpha^{~\beta}=\frac{1}{\lambda_\a\lambda_\b}\phi_\a^{~\b}
,\quad
\widetilde\theta_\alpha^{~j}=\frac{1}{\lambda_\a}\theta_\a^{~j},
$$
while rotation remains $\phi_\a^{~\b}$ unchanged and changes $\theta_\a^{~j}$ to
$$
\widetilde\theta_\alpha^{~j}=\theta_\a^{~k}U_k^{~j}.
$$

Finally, we will use the change of frame given by
\begin{equation*}\label{last-change}
\widetilde Z_\alpha=Z_\alpha,\quad
\widetilde X_j=X_{j} + C_j^{~\beta}Z_\beta,\quad
\widetilde Y_\alpha=Y_\alpha+A_\alpha^{~\beta}Z_\beta+B_\alpha^{~j}X_j
\end{equation*}
such that
$$C_j^{~\alpha}+B_j^{~\alpha}=0$$
and
$$A_\alpha^{~\beta} + \overline{A_\beta^{~\alpha}}
+B_\alpha^{~j}B_j^{~\beta}=0,$$
where
$$B_j^{~\alpha}:=\3\delta_{jk}\overline{B_\alpha^{~k}}.$$
Then the new frame $(\widetilde Z,\widetilde Y,\widetilde X)$ is an $S^{\ell}_{q,p}$-frame and the related one forms $\widetilde\phi_\alpha^{~\beta}$ remain the same, while $\widetilde\theta_\alpha^{~j}$ change to
$$\widetilde\theta_\alpha^{~j}=\theta_\alpha^{~j}-\phi_\alpha^{~\beta}B_\beta^{~j}.$$

\section{universal cycle spaces of $S_{q,p}^\ell$}\label{max-complex}
Assume that $q<\ell.$ In this section, we investigate Grassmannian submanifolds in $S^{\ell}_{q,p}$. We refer \S3 of \cite{Ng} as a reference.
Denote by $N(\ell,m)$ the set of all subspaces $F\subset \mathbb{C}^{p+q}$ such that
$$\langle~, \rangle_{\ell,m}|_F=0,$$
where $\langle~,\rangle_{\ell,m}$ is the hermitian inner product given by \eqref{basic inner}.
For $F\in N(\ell,m)$ and a positive integer $n\geq \dim F$, define
$$\Lambda_F^n:=\{E\in N(\ell,m):\dim E=n, F\subset E\}.$$
Note that since $E\in \Lambda_F^n$ is a null space of $\langle~,\rangle_{\ell,m}$, $n$ should be less or equal to $\ell$.

Let $Z\in S_{q,p}^\ell$ and let $E\in \Lambda_Z^n$. Then $Gr(q, E)$ is a complex submanifold in $S_{q,p}^\ell$, which we call \emph{$n$-confined subgrassmannian} containing $Z$. Choose an $S_{q,p}^\ell$-frame $\{Z_\a, X_j, Y_\a\}$ at $Z$. After a frame change by rotation, we may assume that
$$E=Z+\text{\rm span}\{\3 X_j,~j=1,\ldots,n-q\},$$
where
$$\3 X_j:=X_j+X_{p-q-j+1}.$$
Then $Gr(q, E)$ is an integral manifold of a system
$$\phi_\a^{~\b}=\3\theta_\a^{~j}=\theta_\a^{~k}=0,\quad 1\leq \a,\b\leq q,~ 1\leq j\leq n-q<k\leq p-n$$
with maximal independent condition
\begin{equation}\label{max-indep}
\bigwedge_{\a=1}^q(\theta_\a^{~1}\wedge\cdots\wedge\theta_\a^{~n-q})\neq 0,
\end{equation}
where
$$\3\theta_\a^{~j}:=\theta_\a^{~j}-\theta_\a^{~p-q-j+1},~j=1,\ldots,n-q.$$

Define
$$
   \mathcal{P}^n:=\{(Z,E)\in \mathcal{F}(q,n,{ p+q}):Gr(q, E)\subset S_{q,p}^\ell\}.
$$
Then $\mathcal{P}^n$ is a closed $SU(\ell,m)$-orbit in the flag manifold $\mathcal{F}(q, n, p+q)$ and becomes a fiber bundle over $S_{q,p}^\ell$ under the natural projection defined by $(Z, E)\to Z$.
For $Z\in S^{\ell}_{q,p},$ let
$\mathcal{P}^n_Z$
be the fiber of $\mathcal{P}^n$ over $Z$.
Define a map $F_Z$ by
$$F_Z(Z,E)= \pi_{Z^\perp}(E),$$
where $\pi_{Z^\perp}$ is the orthogonal projection from $Z+X$ to $Z^\perp\subset Z+X$, where $X:=\text{\rm span}\{X_j,~j=1,\ldots,p-q\}$.
Then $F_Z$ is a biholomorphic map between $\{E:(Z, E)\in \mathcal{F}(q,n,p+q)\}$ and $Gr(n-q, Z^\perp)$ sending $\mathcal{P}_Z^n$ onto $S^{\ell-q}_{n-q, p-n}.$

Choose an $S_{q,p}^\ell$-frame $\{Z_\a, X_j, Y_\a\}$ such that
$$Z=\text{\rm span}\{Z_\a,~\a=1,\ldots,q\},\quad E=Z+\text{\rm span}\{\3 X_j,~j=1,\ldots,n-q\}.$$
Then we obtain
$$dZ_\a=\sum_{k=1}^{n-q}\frac12\Check\theta_\a^{~k}\3 X_k
-\sum_{k=1}^{n-q}\frac12\3 \theta_\a^{~k}\Check X_{k}
+\sum_{k=n-q+1}^{p-n}\theta_\a^{~k}X_k
+\phi_\a^{~\b}Y_\b\mod Z,\quad\forall \a$$
$$ d\3 X_j=\sum_{k=1}^{n-q}\3 \omega_j^{~k}\Check X_{k}+
\sum_{k=n-q+1}^{p-n}( \omega_j^{~k}+\omega_{p-q-j+1}^{~k}) X_k
+\3 \theta_j^{~\b}Y_\b\mod E,\quad j=1,\ldots,n-q,$$
where
$$\Check\theta_\a^{~j}=\theta_\a^{~j}+\theta_\a^{~p-q-j+1},$$
$$\3 \omega_j^{~k}=\frac12\left(\omega_j^{~k}+\omega_{p-q-j+1}^{~k}-\omega_j^{~p-q-k+1}-\omega_{p-q-j+1}^{~p-q-k+1}\right)$$
and
$$\Check X_j=X_j-X_{p-q-j+1}.$$
Similar to \S\ref{preli}, we can show that the CR structure of $P^n$ is given by
\begin{equation*}\label{cr-max}
\phi_\a^{~\b}=\3\theta_\a^{~k}=\3\omega_j^{~k}=0,\quad \a,\b=1,\ldots,q,\quad j,k=1,\ldots,n-q.
\end{equation*}
By pulling back a maximal complex submanifold in $S_{q,p}^\ell$ via the projection $(Z,E)\to Z$, we obtain that a maximal complex submanifold of $\mathcal{P}^n$ is a flag manifold $\mathcal{F}(q, n,E)$ for some $\ell$-dimensional space $E\in N(\ell,m)$.

\medskip

For a point $Z\in S^{\ell}_{q,p}$, define
$$\mathcal{C}_Z:=\bigcup_{E\in \Lambda_Z^\ell} T_Z^{1,0} Gr(q,E) .$$
\begin{lemma}\label{not-flat}
$\mathcal{C}_Z$ is not contained in any proper (complex) subspace of $T_Z^{1,0}S^{\ell}_{q,p}$.
\end{lemma}
\begin{proof}
Let
$$\mu=\mu_j^{~\a} \theta_\a^{~j}$$
be a $(1,0)$ form that vanishes on ${\mathcal{C}}_Z$.
To complete the proof, it is enough to show that $\mu= 0$ on $T^{1,0}_ZS_{q,p}^\ell$.
For a given $E\in \Lambda_Z^\ell$, assume that $Gr(q,E)$ is an integral manifold of
$$\phi_\a^{~\b}=\3\theta_\a^{~j}=\theta_\a^{~k}=0,\quad 1\leq \a,\b\leq q,~ 1\leq j\leq\ell-q<k\leq p-\ell.$$
Then by substituting
$$\theta_\a^{~p-q-j+1}=\theta_\a^{~j}-\3\theta_\a^{~j},$$
we obtain
$$\mu=\sum_{j=1}^{\ell-q}(\mu_j^{~\a}+\mu_{p-q-j+1}^{~\a})\theta_\a^{~j}~\mod\{\3\theta_\a^{~j}, \theta_\a^{~k},~1\leq j\leq\ell-q<k\leq p-\ell\}.$$
Therefore if
$\mu=0$ on $T_Z^{1,0}Gr(q,E)$, then by maximal independent condition given in \eqref{max-indep}, we obtain
$$\mu_j^{~\a}+\mu_{p-q-j+1}^{~\a}=0,\quad j=1,\ldots,\ell-q.$$
The same argument for the maximal integral manifold of the system
$$\phi_\a^{~\b}=\theta_\a^{~j}+\theta_\a^{~p-q-j+1}=\theta_\a^{~k}=0,\quad 1\leq \a,\b\leq q,~1\leq j\leq\ell-q<k\leq p-\ell$$
will imply
$$\mu_j^{~\a}-\mu_{p-q-j+1}^{~\a}=0,\quad j=1,\ldots,\ell-q.$$
Hence we obtain
$$\mu_j^{~\a}=\mu^{~\a}_{p-q-j+1}=0,\quad j=1,\ldots,\ell-q.$$
Since $E\in \Lambda_Z$ is arbitrary, we can show that
$$\mu_{j}^{~\a}=0,\quad \forall \a, j,$$
which completes the proof.
\end{proof}

Let $Z_0\in S^{\ell}_{q,p}$ and $E_0\in \Lambda_{Z_0}^\ell$ be fixed and let
$$\mathcal{S}_0:=\{E_0\}.$$
Define an increasing sequence of sets $\mathcal{S}_j$, $j\geq 1$ inductively by
$$\mathcal{S}_j=\bigcup_{F\in \mathcal{S}_{j-1}}\{E\in N(\ell,m):\dim E=\ell,~\dim (E\cap F)\geq \ell-1 \}.$$

\begin{lemma}\label{chain}
$$\bigcup_{E\in \mathcal{S}_{q}} Gr(q,E)=S^{\ell}_{q,p}.$$
\end{lemma}
\begin{proof}
Let $Z_1\in S^{\ell}_{q,p}$. Write
$$Z_1=\text{\rm span}\{W_1,\ldots,W_q\}.$$
After a frame change, we may assume that there exists $\a_0\leq q$ such that
$$E_0\cap Z_1=\text{\rm span}\{W_\a,~ 1\leq \a\leq \a_0\}.$$
Let
$$E_0=\text{\rm span}\{V_1,\ldots,V_{\ell}\}$$
with
$$V_\a=W_\a,\quad \a=1,\ldots,\a_0.$$
By the maximality of $E$, we obtain that for each $\a>\a_0$, there exists $V_{j_\a}$, $\a>\a_0$ such that
$$\langle W_\a, V_{j_\a}\rangle_{\ell,m}\neq 0.$$
After choosing suitable basis of $Z_1$ and $E_0$, we may assume
$$\langle W_\a, V_{j}\rangle_{\ell,m}=\pm\delta_{\a, j},\quad \a_0<\a\leq q.$$
Then a sequence of subspaces $E_\a$  defined by
$$E_\a:=\text{\rm span}\{W_1,\ldots,W_{\a}, V_{\a+1},\ldots,V_\ell\},\quad \a=1,\ldots,q$$
will satisfy
$$E_\a\in \mathcal{S}_{\a}, \quad \a=1,\ldots,q$$
and
$$E_q\supset Z_1.$$
\end{proof}
Lemma~\ref{chain} shows that every two points $Z_0,Z_1$ in $S_{q,p}^\ell$ are connected by a chain of
maximal complex submanifolds $Gr(q, E_\a),~\a=0,\ldots,q$ of $S_{q,p}^\ell$ with $Z_0\in Gr(q, E_0),~Z_1\in Gr(q, E_q)$
such that
$$\dim( E_{\a-1}\cap E_\a)=\ell-1.$$
Similar to Lemma~\ref{chain}, we can also prove the following lemma whose proof we omit.
\begin{lemma}\label{open}
$$\mathcal{C}_{Z_0}=\left\{t\in T^{1,0}_{Z_0}Gr(q,E):E\in \mathcal{S}_{\ell-q}\cap \Lambda_{Z_0}^\ell\right\}.$$
\end{lemma}
\medskip

For integers $q'\leq \ell'\leq (p'+q')/2\leq p'$, define
$$S^{\ell'}_{q',p'}=\{Z'\in Gr(q', p'):\langle ~,\rangle_{\ell',m'}|_{Z'}=0\},$$
where $m'=p'+q'-\ell'$ and
$$\langle u ,v\rangle_{\ell',m'}:=-(u_1\bar v_1+\cdots+u_{\ell'}\bar v_{\ell'})+(u_{\ell'+1}\bar
v_{\ell'+1}+\cdots+u_{p'+q'}\bar v_{p'+q'}).$$
Note that $S^{\ell'}_{q',p'}$ is the unique closed $SU(\ell',m')$ orbit in $Gr(q', p')$.
We shall denote by $\{Z_a', X'_J, Y'_a\}$ an $S_{q', p'}^{\ell'}$-frame and by $\Phi_a^{~b}$, $\Theta_a^{~J}$ the coframes
of $S^{\ell'}_{q',p'}$ corresponding to the coframes $\phi_\a^{~\b},~ \theta_\a^{~j}$ of $S_{q,p}^{\ell}$. In particular, we obtain
$$dZ_a'=\Theta_a^{~J}X'_J+\Phi_a^{~b}Y'_b\mod Z'.$$

Let $f:S^{\ell}_{q,p}\to S^{\ell'}_{q',p'}$ be a CR map. Then $f$ preserves complex submanifolds, i.e. for any complex submanifold $\mathcal{N}\subset S^{\ell}_{q,p}$, there exists a complex submanifold $\mathcal{N}'\subset S^{\ell'}_{q',p'}$ such that $f(\mathcal{N})\subset \mathcal{N}'$.

\begin{lemma}\label{open2}
Let $Z\in S_{q,p}^\ell$ and $f:S^{\ell}_{q,p}\to S^{\ell'}_{q',p'}$ be a germ of a CR embedding at $Z$.
Suppose there exists a maximal complex submanifold $Gr(q',E')\subset S^{\ell'}_{q',p'}$ such that
\begin{equation}\label{send}
f(Gr(q,E))\subset Gr(q', E'),\quad \forall E\in \Lambda_Z^\ell.
\end{equation}
Then
$$f_*(\mathbb{C}T_ZS^{\ell}_{q,p})\subset T^{1,0}_{f(Z)}S_{q',p'}^{\ell'}+T^{0,1}_{f(Z)}S_{q',p'}^{\ell'}.$$
\end{lemma}
\begin{proof}
Assume that $Gr(q',E')$ is an integral manifold of
$$\Phi_a^{~b}=\3\Theta_a^{~J}=\Theta_a^{~K}=0,\quad 1\leq a,b\leq q',~ 1\leq J\leq\ell'-q'<K\leq p'-\ell'$$
where
$$\3\Theta_a^{~J}:=\Theta_a^{~J}-\Theta_a^{~p'-q'-J+1}=0.$$
Let $\mu$ be a one form of $S_{q',p'}^{\ell'}$ in the ideal generated by $\{\Phi_a^{~b},\3\Theta_a^{~J},\Theta_a^{~K},~1\leq a,b\leq q',~ 1\leq J\leq\ell'-q'<K\leq p'-\ell'\}$.
Then \eqref{send} implies
$$\mu(f_*(t))=0,\quad\forall t\in \mathcal{C}_Z.$$
This, together with Lemma~\ref{not-flat}, implies
$$\mu(f_*(t))=0,\quad\forall t\in T^{1,0}_Z S^{\ell}_{q,p},$$
i.e.
\begin{equation}\label{subGr}
f_*(T_Z^{1,0}S^{\ell}_{q,p})\subset T^{1,0}_{f(Z)} Gr(q',E').
\end{equation}
By taking complex conjugation, we obtain
\begin{equation}\label{subGr'}
f_*(T_Z^{0,1}S^{\ell}_{q,p})\subset T^{0,1}_{f(Z)}Gr(q',E') .
\end{equation}

Next we will show that
$$f_*(v)\in T^{1,0}_{f(Z)}S^{\ell'}_{q',p'}+T^{0,1}_{f(Z)}S^{\ell'}_{q',p'}$$
for all $v\in \mathbb{C}T_ZS_{q,p}^\ell$.
Since
$$d\phi_\a^{~\b}=\theta_\a^{~j}\wedge\theta_j^{~\b}~\mod\phi,$$
it is enough to show that
$$\Phi_a^{~b}([f_*(t_j^{~\a}), \overline {f_*( t_j^{~\b})}])=0
$$
at $f(Z)$,
where $\{t_j^{~\a}\}$ are $(1,0)$ vector fields dual to $\{\theta_\a^{~j}\}$.
Since
$$\Phi_a^{~b}\left(\left[f_*(t_j^{~\a}),\overline{f_*(t_j^{~\b})}\right]\right)=d\Phi_a^{~b}\left(f_*(t_j^{~\a}),\overline{f_*(t_j^{~\b})}\right)$$
and
$$
d\Phi_a^{~b}=\sum_{1\leq J\leq \ell'-q'}\left(\3\Theta_a^{~J}\wedge\Theta_J^{~b}+\Theta_a^{~J}\wedge\3\Theta_J^{~b}\right)
+\sum_{\ell'-q'<J\leq m'-q'}\Theta_a^{~J}\wedge\Theta_J^{~b},$$
\eqref{subGr} and \eqref{subGr'} will imply
$$\Phi_a^{~b}([f_*(t_j^{~\a}), \overline {f_*( t_j^{~\b})}])=0$$
at $f(Z)$.
Hence the conclusion follows.
\end{proof}

\begin{definition}\label{cr-trans}
A germ of a CR map $f:(S^{\ell}_{q,p},Z)\to (S^{\ell'}_{q',p'},Z')$ is said to be \emph{transversal} if
$$ f_*(t)\not \in T_{Z'}^{1,0}S_{q',p'}^{\ell'}+T_{Z'}^{0,1}S_{q',p'}^{\ell'},\quad \forall t\not \in T_Z^{1,0}S_{q,p}^\ell+T_Z^{0,1}S_{q,p}^\ell.$$
\end{definition}

\section{$S^{\ell'}_{q',p'}$-frames adapted to $f$}
From now on, we shall follow the index convention that
small Greek indices $\a,\b,\g$ run over $\{1,\ldots,q\}$,
small Latin indices $i,j,k$ over $\{1,\ldots,\ell-q,(p'-q')-(m-q)+1,\ldots,p'-q'\}$,
small Latin indices $a,b,c,d$ over $\{1,\ldots, q'\}$
and large Latin indices $I, J, K$ over $\{1,\ldots,p'-q'\}$. For simplicity, we shall abuse the notation by writing
$\Sigma$ instead of $f^{*}\Sigma$ for any form $\Sigma$ on $S^{\ell'}_{q',p'}$.
We also use the notation
$\phi_{a}^{~b}=0$ if $a>q$ or $b>q$,
$\theta_{a}^{~J}=0$ if $a>q$ or $\ell-q<J\leq (p'-q')-(m-q)$ and
$\{\theta_\a^{~j}, ~1\leq \a\leq q,~j\leq \ell-q\text{ or }j>(p'-q')-(m-q)\}$ is a basis of $(1,0)$ forms satisfying
$$d\phi_\a^{~\b}=-\sum_{j\leq \ell-q}\theta_\a^{~j}\wedge\1{\theta_\b^{~j}}+\sum_{j>(p'-q')-(m-q)}\theta_\a^{~j}\wedge\1{\theta_\b^{~j}}~\mod \phi.$$

In this section we use the structure equation
for $\phi_\a^{~\b}$ modulo the ideal $\phi$ generated by the contact forms $\phi_{\a}^{~\b}$, i.e. the equations
\begin{equation}
d\phi_{a}^{~b} = \theta_{a}^{~J}\wedge \theta_{J}^{~b} \mod \phi, \quad
d\Phi_{a}^{~b} = \Theta_{a}^{~J}\wedge \Theta_{J}^{~b} \mod \phi. \Label{seq}
\end{equation}
Note that due to our convention, both sides of the first equation are zero
if $a>q$ or $b>q$ and for the same reason the summation is only performed over $J\in \{1,\ldots,\ell-q,(p'-q')-(m-q)+1,\ldots,p'-q'\}$.

\subsection{Determination of $\Phi_{1}^{~1}$}
For a CR map $f:S^{\ell}_{q,p}\to S^{\ell'}_{q',p'}$,
the pullback of $\Phi_{a}^{~b}$ via $f$ is a linear combination of $\phi_{\a}^{~\b}$ and the pull back of $\Theta_a^{~J}$via $f$
is a linear combination of $\theta_\a^{~j}$ modulo $\phi.$
Consider the diagonal terms $\Phi_a^{~a}$, $a=1,\ldots,q'$.
Suppose that (the pullbacks of) $\Phi_a^{~a}$
vanish identically for all $a$.
Consider the structure equation
\begin{equation*}\label{levi-id}
0=d\Phi_1^{~1}=-\sum_{J\leq \ell'-q'}\Theta_1^{~J}\wedge\1{\Theta_1^{~J}}+\sum_{J> \ell'-q'}\Theta_1^{~J}\wedge\1{\Theta_1^{~J}}\quad\mod \phi.
\end{equation*}
By substituting
$$\Theta_1^{~J}=\sum_{\a, j}h^{~J}_{\a, j}\theta_\a^{~j}~\mod \phi,$$
we obtain
\begin{equation}\label{h-sp}
-\sum_{J\leq \ell'-q'}h^{~J}_{\a,j}\overline{h^{~J}_{\b,k}}+\sum_{J>\ell'-q'}h^{~J}_{\a,j}\overline{h^{~J}_{\b,k}}=0,~\forall \a, \b, j, k.
\end{equation}
For each $\a$ and $j$, define a vector $h_{\a,j}$ by
$$h_{\a,j}:=(h^{~J}_{\a,j})_{J=1,\ldots,p'-q'}.$$
Then \eqref{h-sp} implies
that the space span$\{h_{\a,j}\}_{\a,j}$ is a null space of the inner product $\langle~,\rangle'$
on $\mathbb{C}^{p'-q'}$
defined by
$$\langle u,v\rangle':=-\sum_{J\leq \ell'-q'}u^J\overline{v^J}+\sum_{J> \ell'-q'}u^J\overline{v^J}.$$
After a suitable frame change by rotation, we may assume that on an open set,
$$\text{\rm span}\{h_{\a,j}\}_{\a,j}=\text{\rm span}\{\3 X_J',~1\leq J\leq n_1-q'\}$$
for some $n_1\leq \ell',$ where
$$\3 X_J':=X_J'+X_{p'-q'-J+1}',$$
i.e. $f(S^{\ell}_{q,p})$ is an integral manifold of the system
\begin{equation}\label{integral}
\3\Theta_1^{~J}=0~\mod\phi, ~1\leq J\leq n_1-q'.
\end{equation}

Fix a point $Z\in S^{\ell}_{q,p}$. For a given $E\in \Lambda_Z^\ell$, choose a minimal $Gr(q',E')\subset S^{\ell'}_{q',p'}$ through $f(Z)$ such that
$$f(Gr(q,E))\subset  Gr(q, E').$$
Then \eqref{integral} implies that $Gr(q',E')$ is an integral manifold of
$$\3\Theta_1^{~J}=0,~1\leq J\leq n_1-q'.$$
Hence we obtain
$$E'\subset f(Z)+\text{\rm span}\{\3 X'_J,X_K',~1\leq J\leq n_1-q'<K\leq p'-n_1\}.$$
This implies that $Gr(q',E')$ is an integral manifold of
\begin{equation}\label{Theta-a}
\3\Theta_a^{~J}=0,\quad 1\leq a\leq q',~1\leq J\leq n_1-q'.
\end{equation}
Since $E\in \Lambda_{Z}^\ell$ is arbitrary, together with Lemma~\ref{not-flat}, we can conclude that
\eqref{Theta-a} holds
on $T_Z^{1,0} S^{\ell}_{q,p}$.

Now consider
$$
0=d\Phi_2^{~2}=-\sum_{J\leq \ell'-q'}\Theta_2^{~J}\wedge\1{\Theta_2^{~J}}+\sum_{J> \ell'-q'}\Theta_2^{~J}\wedge\1{\Theta_2^{~J}}
\quad\mod \phi.
$$
By \eqref{Theta-a},
we obtain
$$
0=-\sum_{n_1-q'<J\leq \ell'-q'}\Theta_2^{~J}\wedge\1{\Theta_2^{~J}}+\sum_{ \ell'-q'<J\leq p'-n_1}\Theta_2^{~J}\wedge\1{\Theta_2^{~J}}
\quad\mod \phi.
$$
Then similar to the case of $\Theta_1^{~J}$, we may assume that
$$\3\Theta_2^{~J}=\Theta_2^{~K}=0~\mod\phi, ~1\leq J\leq n_2-q'<K\leq p'-n_2$$
for some $n_2\geq n_1$ and
\begin{equation*}
\3\Theta_a^{~J}=0,\quad 1\leq a\leq q',~1\leq J\leq n_2-q'
\end{equation*}
on $T_Z^{1,0} S^{\ell}_{q,p}$.
By continuing this process, we can show that there exists an integer $n_{q'}$ such that
$$\3\Theta_a^{~J}=\Theta_a^{~K}=0,~1\leq a\leq q',~1\leq J\leq n_{q'}-q'<K\leq p'-n_{q'}$$
on $T_Z^{1,0} S^{\ell}_{q,p}$.

Define
$$ F':=f(Z)+\text{\rm span}\{\3 X_J',~1\leq J\leq n_{q'}-q'\}.$$
Then we have
$$f_*(T_{Z}^{1,0}Gr(q,E))\subset T_{f(Z)}^{1,0} Gr(q,  F'),~\forall E\in \Lambda_{Z}^\ell.$$
By Lemma~\ref{open2}, we obtain
$$f_*(\mathbb{C}T_{Z}S^{\ell}_{q,p})\subset T_{f(Z)}^{1,0}S^{\ell'}_{q',p'}+T_{f(Z)}^{0,1}S^{\ell'}_{q',p'},$$
which contradicts our assumption on the transversality of $f$.
\medskip

Hence there exists at least one diagonal
term of $\Phi$ whose pullback does not vanish identically.
Choose such a diagonal term of $\Phi$, say $\Phi_1^{~1}$.
Then on an open set, $\Phi_1^{~1}\neq 0.$
Since the pullback of $\Phi_{1}^{~1}$ to $S^{\ell}_{q,p}$ is a contact form, we can write
$$\Phi_1^{~1}= c_\alpha^{~\beta}\phi^{~\alpha}_{\beta}$$
for some smooth functions $c_\alpha^{~\beta}$. Since
$(\phi_\alpha^{~\beta})$ and $(\Phi_a^{~b})$ are antihermitian, the
matrix $(c_\alpha^{~\beta})$ is hermitian. Then there exists
a change of frame on $S^{\ell}_{q,p}$ (change of position in Definition~\ref{changes}) given by
$$\widetilde Z_\alpha=U_\alpha^{~\beta}Z_\beta,
~\widetilde Y_\alpha=U_\alpha^{~\beta}Y_\beta,~\widetilde X_j=X_j$$ for some unitary matrix $U$
such that $c_{\a}^{~\b}$ is diagonalized and hence
the new contact forms $\phi_\alpha^{~\beta}$, $\alpha,\beta=1,\ldots,q$, satisfy
\begin{equation*}
\Phi_1^{~1}=\sum_{\alpha=1}^rc_\alpha \phi_\alpha^{~\alpha},
\quad 1\le r\le q,
\end{equation*}
where $c_\alpha$, $\alpha=1,\ldots,r$, are nonzero real valued smooth functions.
After dilation of $\Phi_1^{~1}$, we may further assume that $$c_1=\pm 1.$$

Denote by $\theta_\alpha$
the ideal generated by
$\{\theta_\alpha^{~j},~\forall j\}$.
We will prove the following.
\begin{lemma}\Label{r}
Assuming $(\ell'-q')<2(\ell-q)$, we have  $r=1$ and after suitable frame changes, we obtain that either
$\ell-q\leq \ell'-q'$, $m-q\leq m'-q'$ and $f$ satisfies
\begin{align}
\Phi_{1}^{~1}&=\phi_{1}^{~1},\Label{phi11-norm} \\
\Theta_1^{~J}-\theta_{1}^{~J}&=0 \mod \phi,\quad 1\leq J\leq \ell-q\text{  or  }2(\ell'-q')-(\ell-q)<J\leq p'-q' \Label{theta-norm}\\
\Theta_1^{~J}&=0 \mod \phi,\{\theta_\a,\a\geq 2\},\quad \ell-q<J\leq 2(\ell'-q')-(\ell-q)\label{theta1},\\
\3\Theta_1^{~J}&=0\mod\phi,\quad \ell-q<J\leq \ell'-q'\label{theta1-hat}
\end{align}
or $\ell-q\leq m'-q'$, $m-q\leq \ell'-q'$ and $f$ satisfies
\begin{align}
\Phi_{1}^{~1}&=-\phi_{1}^{~1},\Label{phi11-norm'} \\
\Theta_1^{~J}-\theta_{1}^{~p'-q'-J+1}&=0 \mod \phi,\quad 1\leq J\leq m-q\text{  or  }2(\ell'-q')-(m-q)<J\leq p'-q' \Label{theta-norm'}\\
\Theta_1^{~J}&=0 \mod \phi,\{\theta_\a,\a\geq 2\},\quad m-q<J\leq 2(\ell'-q')-(m-q)\label{theta1'}\\
\3\Theta_1^{~J}&=0\mod\phi,\quad m-q<J\leq \ell'-q',\label{theta1-hat'}
\end{align}
where
$$\3\Theta_1^{~J}:=\Theta_1^{~J}-\Theta_1^{~2(\ell'-q')-J+1}.$$
In addition, we can choose a frame such that
for each $a\geq 2,$ there exist smooth functions $\mu_{a}^{~\b}$, $\b\geq 2$ satisfying
either
\begin{equation}
\Theta_a^{~j}=\sum_{\b\geq 2}\mu_{a}^{~\b}\theta_\b^{~j} \mod \phi, \quad j\leq \ell-q\text{  or  }j>(p'-q')-(m-q)\label{theta-a-j},
\end{equation}
if $\Phi_1^{~1}=\phi_1^{~1}$
or
\begin{equation}
\Theta_a^{~p'-q'-j+1}=\sum_{\b\geq 2}\mu_{a}^{~\b}\theta_\b^{~j} \mod \phi, \quad j\leq \ell-q\text{  or  }j>(p'-q')-(m-q)\label{theta-a-j'},
\end{equation}
if $\Phi_1^{~1}=-\phi_1^{~1}$.

\end{lemma}

\bpf
The structure equations \eqref{seq} for $\Phi_1^{~1}$ yield
\begin{equation}\Label{phi11}
\sum_{J=1}^{\ell'-q'} \Theta_{1}^{~J}\wedge \1{\Theta_{1}^{~J}}-\sum_{J=\ell'-q'+1}^{p'-q'} \Theta_{1}^{~J}\wedge \1{{\Theta}_{1}^{~J}}
= \sum_{\a} c_{\a}\left(\sum_{j=1}^{\ell-q} \theta_{\a}^{~j}\wedge \1{\theta_{\a}^{~j}} -\sum_{j=\ell-q+1} ^{p-q} \theta_{\a}^{~j}\wedge \1{{\theta}_{\a}^{~j}}\right)\mod \phi.
\end{equation}
Let
\begin{equation}\Label{h-eq}
\Theta_1^{~J}=\sum_{\a,j}h^{~J}_{\alpha,j}\theta_\alpha^{~j}\text{   mod   }~\phi
\end{equation}
and define
$$h_{\a,j}:=(h_{\a,j}^{~J})_{J=1,\ldots,p'-q'},\quad \forall \a,j.$$
Then \eqref{phi11} implies
\begin{equation}\label{orth}
\langle h_{\a,j},h_{\b,k}\rangle'
=c_\alpha\delta_{\alpha\beta}\cdot\3\delta_{jk},\quad \forall \a,\b, j,k,
\end{equation}
where
$$c_{\a}:=0, \quad\a>r.$$
Thus the vectors $h_{\a,j}$
are pairwise orthogonal and have length $c_{\a}$ independent of $j$ with respect to $\langle~,\rangle'$.
After a suitable
rotation,
we may assume that either
$$\text{\rm span}\{h_{1,j}\}_j=\text{\rm span}\{X_j'\}_j$$
and therefore $\ell-q\leq \ell'-q'$ and $m-q\leq m'-q'$ if $c_1=1$ or
$$\text{\rm span}\{h_{1,j}\}_j=\text{\rm span}\{X'_{p'-q'-j+1}\}_j$$
and therefore
$\ell-q\leq m'-q'$ and $m-q\leq \ell'-q'$
if $c_1=-1$. This implies either
$$\Theta_1^{~J}-\theta_{1}^{~J}=0 \mod \phi,\{\theta_\a,\a\geq 2\}$$
if $c_1=1$ or
$$\Theta_1^{~J}-\theta_{1}^{~p'-q'-J+1}=0 \mod \phi,\{\theta_\a,\a\geq 2\}$$
if $c_1=-1$.
This together with \eqref{orth} implies either
$$\Theta_1^{~j}-\theta_{1}^{~j}=0 \mod \phi$$
if $c_1=1$ or
$$\Theta_1^{~j}-\theta_{1}^{~p'-q'-j+1}=0 \mod \phi$$
if $c_1=-1$.

 Now fix $\a>1$. Then the vector space span$\{h_{\a,j}\}_j$ is in the orthogonal complement of $\text{span}\{h_{1,j}\}_j$ with respect to $\langle~,\rangle'$. Note that $\langle~,\rangle'$ restricted to an orthogonal space of $\text{span}\{h_{1,j}\}_j$ has $(m'-q')-(m-q)$ positive eigenvalues and $(\ell'-q')-(\ell-q)$ negative eigenvalues. Therefore, the maximal null space spanned by $\{h_{\a,j}\}_j$ is at most $(\ell'-q')-(\ell-q)$ dimensional. Since $\ell'-q'<2(\ell-q)$ by our assumption, this is only possible when
$c_\a=0$ for all $\a\ne1$, i.e. $r=1$ and
$$\langle h_{\a,j},h_{\b,k}\rangle'
=0,~\a,\b\geq 2.$$
In particular, either \eqref{phi11-norm} or \eqref{phi11-norm'} holds and span$\{h_{\a,j}\}_{\a\geq 2,j}$ is a null space of $\langle~,\rangle'$.
Then after a frame change by rotation, we may assume that either
$$\text{\rm span}\{h_{\a,j}\}_{\a\geq 2,j}=\text{\rm span}\{ X_J'+X'_{2(\ell'-q')-J+1},~\ell-q<J\leq \ell'-q'\}$$
if $c_1=1$
or
$$\text{\rm span}\{h_{\a,j}\}_{\a\geq 2,j}=\text{\rm span}\{X_J'+X'_{2(\ell'-q')-J+1},~m-q<J\leq \ell'-q'\}$$
if $c_1=-1$ and hence either \eqref{theta-norm}, \eqref{theta1} and \eqref{theta1-hat} or \eqref{theta-norm'}, \eqref{theta1'} and \eqref{theta1-hat'} hold.

For further adaptation, fix $a\ge 2$ and let
\begin{equation}\Label{la}
\Phi_a^{~1}=\lambda_\b^{~\g}\phi_\g^{~\b}
\end{equation}
for some smooth functions $\lambda_\b^{~\g}$. First assume
$$\Phi_1^{~1}=\phi_1^{~1}.$$
Then \eqref{seq}, \eqref{theta-norm}, \eqref{theta1} and \eqref{theta1-hat}  imply
\begin{equation}\Label{th-aj}
\sum_{j}\Theta_a^{~j}\wedge \theta_j^{~1}
=\lambda_1^{~\g}\left(\theta_\g^{~j}\wedge\theta_j^{~1}\right),~\mod \phi,\{ \1{\theta}_{\b},~\b\ge2\}.
\end{equation}
There exists a change of position that leaves $\Theta_{1}^{~J}$ invariant and replaces
$\Theta_a^{~J}$
with $\Theta_a^{~J}-\lambda_1^{~1}\Theta_1^{~J}$ for given $a\ge2$.
This change of position leaves $\Phi_{1}^{~1}$ invariant
and transforms $\Phi_{a}^{~1}$ into $\Phi_{a}^{~1}-\l_{1}^{~1}\Phi_{1}^{~1}$ for given $a\ge2$.
After performing such change of position, \eqref{th-aj} becomes
\begin{equation*}
\sum_{j}\Theta_a^{~j}\wedge \theta_j^{~1}=\sum_{\g\geq 2}\lambda_1^{~\g}\left(\theta_\g^{~j}\wedge\theta_j^{~1}\right),~\mod \phi,\{ \1{\theta}_{\b},~\b\ge2\}.
\end{equation*}
Since $\Theta_a^{~J}$ are $(1,0)$ forms but $\theta_j^{~1}$ are $(0,1)$ forms and linearly independent,
it follows that for each fixed $a\ge2$, we have
\begin{equation*}\Label{theta-aj}
\Theta_a^{~j}=\sum_{\g\geq 2}\lambda_1^{~\g}\theta_\g^{~j} \mod \phi,
\end{equation*}
i.e. \eqref{theta-a-j} holds, where we let $\mu_a^{~\g}=\l_1^{~\g}.$ Similar argument for $\Phi_1^{~1}=-\phi_1^{~1}$ will show that \eqref{theta-a-j'} holds, which completes the proof.
\epf

By \eqref{theta1-hat} or \eqref{theta1-hat'} we may assume that $f$ satisfies either
\begin{equation}\label{indep1}
\Theta_1^{~\ell-q+1}\wedge\cdots\wedge\Theta_1^{~\ell-q+n_1}\neq 0,
~\Theta_1^{~\ell-q+n_1+1}=\cdots=\Theta_1^{~2(\ell'-q')-n_1-1}=0
\end{equation}
for some $n_1\leq \ell'-q'-(\ell-q)$ if $\Phi_1^{~1}=\phi_1^{~1}$
or
$$\Theta_1^{~m-q+1}\wedge\cdots\wedge\Theta_1^{~m-q+n_1}\neq 0,
~~\Theta_1^{~m-q+n_1+1}=\cdots=\Theta_1^{~2(\ell'-q')-n_1-1}=0
$$
for some $n_1\leq \ell'-q'-(m-q)$ if $\Phi_1^{~1}=-\phi_1^{~1}$.

\subsection{Determination of $\Phi_{2}^{~2}$ and $\Phi_{2}^{~1}$}
Next for each fixed $a\ge2$, let
\begin{equation}\Label{lab}
\Phi_a^{~a}=\lambda_{\beta}^{~\g}\phi_\g^{~\beta}
\end{equation}
Using the structure equation for $
\Phi_a^{~a}$
together with \eqref{theta-a-j} or \eqref{theta-a-j'} in Lemma~\ref{r}, we obtain
either
\begin{equation}\Label{2a}
\sum_{J=\ell-q+1}^{2(\ell'-q')-(\ell-q)} \Theta_a^{~J}\wedge \Theta^{~a}_J
=\lambda_{\beta}^{~1}\left(\theta_1^{~j}\wedge\theta_j^{~\b}\right)
\mod \{\theta_\g: \g\ge2\},\phi
\end{equation}
or
\begin{equation}\Label{2a'}
\sum_{J=m-q+1}^{2(\ell'-q')-(m-q)} \Theta_a^{~J}\wedge \Theta^{~a}_J
=\lambda_{\beta}^{~1}\left(\theta_1^{~j}\wedge\theta_j^{~\b}\right)
\mod \{\theta_\g: \g\ge2\},\phi.
\end{equation}
Suppose $\l_{\b}^{~1}\ne0$ for some $\b$.
Since $\ell-q\leq m-q$, on the left-hand side of \eqref{2a} or \eqref{2a'} we have at most $\ell'-q'-(\ell-q)$ dimensional null spaces,
whereas on the right-hand side we have $(\ell-q)$ dimensional null spaces.
Since $\ell'-q'<2(\ell-q)$, this is impossible.
Hence we have $\l_{\b}^{~1}=0$ for all $\b$ and therefore \eqref{lab} becomes
$$\Phi_a^{~a}=0~\text{   mod
}\{\phi_\g^{~\beta} : \g\ge 2\}, \quad a\ge2.$$
Since $\Phi_{a}^{~b}$ and $\phi_{\a}^{~\b}$ are antihermitian, we also have
$$\Phi_a^{~a}=0~\text{   mod
}\{\phi_\g^{~\beta} : \b\ge 2\}, \quad a\ge2,$$
and hence
\begin{equation}\Label{phi-aa}
\Phi_a^{~a}=0~\text{   mod
}\{\phi_\g^{~\beta}: \b,\g\ge2 \}, \quad a\ge 2.
\end{equation}
\medskip

We will repeat the argument from the beginning of this section. Assume first that
$\Phi_{a}^{~a}=0$ for all $a\ge 2$.
Similarly, we can choose $Gr(q', F')\subset S^{\ell'}_{q',p'}$ such that
$$f_*\left(T^{1,0}_ZGr(q,E)\cap \text{ kernel}\{\theta_1\}\right)\subset T^{1,0}_{f(Z)}S^{\ell'}_{q',p'},~\forall E\in \Lambda_Z^\ell$$
and hence $f_*$ restricted to $T_Z^{1,0} (S^{\ell}_{q,p}\cap \{Z_1=\text{constant}\})$ contradicts our assumption on the transversality.
Therefore we can choose a non trivial $\Phi_{a}^{~a}$ for some $a$, say $a=2$.
Then \eqref{phi-aa} implies that, after a change of position as before,
we may assume
$$\Phi_{2}^{~2}=\sum_{\a\ge 2} c_{\a} \phi_{\a}^{~\a}$$
for some real $c_{\a}$ not all zero and
\eqref{seq} yields
\begin{equation*}\label{two theta1}
\Theta_2^{~J}\wedge \Theta_J^{~2}=\sum_{\alpha\ge2} c_\alpha\theta_\alpha^{~j}\wedge\theta_j^{~\alpha}
~\text{   mod   }~\phi.
\end{equation*}
Since the proof of Lemma~\ref{r} can be repeated for $\Phi_{2}^{~2}$
instead of $\Phi_{1}^{~1}$, we conclude that
the only one $c_{\a}$, say $c_{2}$ can be different from zero. After a dilation fixing $\Phi_1^{~1}$, we may assume
$$\Phi_{2}^{~2}= c_2\phi_{2}^{~2},$$
where $c_2=\pm 1$
and
\begin{equation}\Label{theta2}
\T_{2}^{~J}\wedge\T_{J}^{~2} = c_2\theta_{2}^{~j}\wedge \theta_{j}^{~2}
\mod\phi.
\end{equation}

We claim that $c_2=c_1$. First assume that $c_1=1$, i.e. $\Phi_1^{~1}=\phi_1^{~1}$.
Since
\begin{equation*}\Label{tet2}
\Phi_{2}^{~1}=\eta_{\b}^{~\a} \phi_{\a}^{~\b}
\end{equation*}
for suitable $\eta_{\b}^{~\a}$, we obtain
\begin{equation*}
\T_{2}^{~J}\wedge \T_{J}^{~1} = \eta_{\b}^{~\a} \theta_{\a}^{~j} \wedge
\theta_{j}^{~\b} \mod \phi,
\end{equation*}
which in view of Lemma~\ref{r}, yields
$$\sum_{j}\Theta_2^{~j}\wedge\theta_j^{~1}+\sum_{\ell-q<J\leq \ell'-q'}\3 \Theta_2^{~J}\wedge\Theta_J^{~1}=\eta_{\b}^{~\a} \theta_{\a}^{~j} \wedge
\theta_{j}^{~\b} \mod \phi.
$$
By Lemma~\ref{r}, we can substitute
\begin{equation}\Label{teta2-all}
\Theta_2^{~j}=\mu^{\a}\theta_\a^{~j}\mod \phi
\end{equation}
for some $\mu^\a$ with $\mu^1=0$, which
yields
\begin{equation}\label{h-id-hat}
\sum_{\ell-q<J\leq \ell'-q'}\3 \Theta_2^{~J}\wedge\Theta_J^{~1}=\sum_{\a}(\eta_1^{~\a}-\mu^\a)\theta_\a^{~j}\wedge\theta_j^{~1}+\sum_{\a}\sum_{\b\geq 2}
\eta_{\b}^{~\a} \theta_{\a}^{~j} \wedge
\theta_{j}^{~\b} \mod \phi.
\end{equation}
Since the left-hand side of \eqref{h-id-hat} contains at most $(\ell'-q')-(\ell-q)$ linearly independent $(1,0)$ forms while the right-hand side contains at least $(p-q)$ linearly independent $(1,0)$ forms unless trivial, we obtain that $\eta^{~\a}_1=\mu^\a$ and $\eta_{\b}^{~\a}=0$ for all $\b>1,$ i.e.
$$\Phi_2^{~1}=\sum_{\a\ge 2}\mu^\a\phi_\a^{~1}$$
and
$$\sum_{\ell-q<J\leq \ell'-q'}\3 \Theta_2^{~J}\wedge\Theta_J^{~1}=0\mod\phi.$$
Then by condition \eqref{indep1} we obtain
\begin{equation}\label{teta3-hat}
\3\Theta_2^{~J}=0\mod\phi,~\ell-q<J\leq \ell-q+n_1.
\end{equation}

Substituting \eqref{teta2-all} and \eqref{teta3-hat} into \eqref{theta2}, we obtain
$$\sum_{\a\geq 2}|\mu^\a|^2\theta_\a^{~j}\wedge\theta_j^{~\a}+\sum_{\ell-q+n_1<J\leq p'-q'-(m-q)}\Theta_2^{~J}\wedge\Theta_J^{~2}=c_2\theta_2^{~j}\wedge\theta_j^{~2}\mod\phi,\{\theta_\a^{~j}\wedge\theta_j^{~\b},\a\neq \b\}.$$
Since $\sum_{\ell-q+n_1<J\leq p'-q'-(m-q)}\Theta_2^{~J}\wedge\Theta_J^{~2}$ has at most $(\ell'-q')-(\ell-q+n_1)$ dimensional null spaces while
$\theta_\a^{~j}\wedge\theta_j^{~\a}$ has at least $(\ell-q)$-dimensional null spaces, we obtain $|\mu^{2}|^2=c_2>0$, $\mu^\a=0$ for all $\a\geq 3$
and
\begin{equation}
\sum_{\ell-q+n_1<J\leq p'-q'-(m-q) }\Theta_2^{~J}\wedge\Theta_J^{~2}=0\mod\phi.\label{Theta_2-vanish}
\end{equation}
In particular, we have
$$\Phi_2^{~1}=\mu\phi_2^{~1},~\Phi_2^{~2}=\phi_2^{~2}$$
and
$$\Theta_2^{~j}=\mu\theta_2^{~j}\mod\phi$$
for some $\mu$ with $|\mu|^2=1.$
After a change of position given by $\widetilde Z_2=uZ_2$ for some $u\in \mathbb{C}$ with $|u|=1$, we may assume $\mu=1$, i.e.
we have
$$\Phi_2^{~1}=\phi_2^{~1}$$
and
$$\Theta_2^{~j}=\theta_2^{~j}\mod\phi.$$

Finally, as in the proof of Lemma~\ref{r} we substitute
$$\Theta_2^{~J}=\sum_{\a, j}h^{~J}_{\a,k}\theta_\a^{~k}\mod\phi,\quad \ell-q+n_1< J\leq p'-q'-(m-q)$$
into \eqref{Theta_2-vanish}
to obtain
$$-\sum_{J=\ell-q+n_1+1}^{\ell'-q'} h_{\a,j}^{~J}\1{h^{~J}_{\b,k}}+ \sum_{J=\ell'-q'+1}^{p'-q'-(m-q)} h_{\a,j}^{~J}\1{h^{~J}_{\b,k}}=0.$$
In view of \eqref{teta3-hat}, after a suitable rotation fixing $\Theta_a^{~J}, ~\3\Theta_a^{~J}$ for $1\leq J\leq \ell-q+n_1$ and $\Theta_a^{~J}$ for $J>2(p'-q')-(m-q)$, we may assume that
\begin{equation*}\label{theta-aJ}
\3\Theta_2^{~J}=\Theta_2^{~K}=0\mod\phi,\quad \ell-q+n_1< J\leq \ell-q+n_2,~2(\ell'-q')-(\ell-q)<K\leq p'-q'-(m-q)
\end{equation*}
with independence condition
$$\Theta_2^{~\ell-q+n_1+1}\wedge\cdots\wedge\Theta_2^{~\ell-q+n_2}\neq 0,
~\Theta_2^{~\ell-q+n_2+1}=\cdots=\Theta_2^{~2(\ell'-q')-n_2-1}=0$$
for some $n_2\geq n_1.$
The same argument can be applied for the case of $\Phi_1^{~1}=-\phi_1^{~1}$ as well to determine $\Phi_2^{~2}, \Phi_2^{~1}$ and $\Theta_2^{~J}$.

\subsection{Determination of $\Phi_{a}^{~b}$ for $a,b>2$}
Now we will repeat again the arguments,
where we replace $1$ by $2$ and $2$ by $3$ and so on.
Then by induction argument on $a$, we
obtain the following lemma.

\begin{lemma}\Label{eds}
For any CR-embedding $f\colon S^{\ell}_{q,p}\to S^{\ell'}_{q',p'}$
there is a choice of $S^{\ell}_{q,p}$-frame and $S^{\ell'}_{q',p'}$-frame
such that the pulled back forms satisfy either
\begin{align}
\Phi_a^{~b}-\phi_a^{~b}&=0,\nonumber\\
\Theta_a^{~J}-\theta_a^{~J}&=0 \mod\phi,~1\leq J\leq \ell-q~\text{or }2(\ell'-q')-(\ell-q)<J\leq p'-q',\label{pre-last}\\
\3\Theta_a^{~J}&=0\mod\phi,~\ell-q<J\leq \ell'-q'\label{pre-last-eq}
\end{align}
or
\begin{align*}
\Phi_a^{~b}+\phi_a^{~b}&=0,\nonumber\\
\Theta_a^{~J}-\theta_a^{~p'-q'-J+1}&=0 \mod\phi,~1\leq J\leq m-q\text{  or  }2(\ell'-q')-(m-q)<J\leq p'-q' ,\\
\3\Theta_a^{~J}&=0\mod\phi,~m-q<J\leq \ell'-q'
\end{align*}
for all $a,b=1,\ldots,q'$,
where
$$\3\Theta_a^{~J}:=\Theta_a^{~J}-\Theta_a^{~2(\ell'-q')-J+1}.$$
\end{lemma}

\subsection{Determination of $\Theta_a^{~J}$,$\3\Theta_a^{~J}$}
Assume that $\Phi_a^{~b}=\phi_a^{~b}$. Let
\begin{align}
\Theta_a^{~J}-\theta_a^{~J}&=\eta_{a,\b}^{~J,\a}\phi_\a^{~\b},~1\leq J\leq \ell-q~\text{or }2(\ell'-q')-(\ell-q)<J\leq p'-q',\label{theta-null1} \\
\3\Theta_a^{~J}&=\3\eta_{a,\b}^{~J,\a}\phi_\a^{~\b},~\ell-q<J\leq \ell'-q'\label{theta-null2}.
\end{align}
By differentiating \eqref{theta-null1} using the structure equation and substituting \eqref{pre-last} and \eqref{pre-last-eq},
we obtain
\begin{equation}\label{theta-null1-n}
\theta_a^{~k}\wedge(\Omega_k^{~J}-\omega_k^{~J})
+\sum_{\ell-q<K\leq\ell'-q'}\Theta_a^{~K}\wedge(\Omega_K^{~J}+\Omega_{2(\ell'-q')-K+1}^{~J})
=\eta_{a,\b}^{~J,\a}\theta_\a^{~k}\wedge\theta_k^{~\b}\mod \phi
\end{equation}
for $1\leq J\leq \ell-q$ or $2(\ell'-q')-(\ell-q)<J\leq p'-q'.$

Suppose $(\eta_{a,\b}^{~J,\a})\neq 0$ for some $a>q$.
In view of Lemma~\ref{eds}, the left-hand side of \eqref{theta-null1-n} has at most $(\ell'-q')-(\ell-q)$ linearly independent $(1,0)$ forms while
the right-hand side has at least $(p-q)$ linearly independent $(1,0)$ forms.
Hence under the condition $\ell'-q'<2(\ell-q)$, we obtain
$$\eta_{a,\b}^{~J,\a}=0,\quad\forall a>q,$$
i.e. for $a>q,$
\begin{equation*}\label{theta-null1'}
\Theta_a^{~J}=0,\quad 1\leq J\leq \ell-q~\text{or }2(\ell'-q')-(\ell-q)<J\leq p'-q'.
\end{equation*}

Let $\a\leq q$ and $J\leq \ell-q$ or $J>2(\ell'-q')-(\ell-q).$ Then \eqref{theta-null1-n} becomes
$$\sum_{\ell-q<K\leq\ell'-q'}\Theta_\a^{~K}\wedge(\Omega_K^{~J}+\Omega_{2(\ell'-q')-K+1}^{~J})
=\sum_{\g\neq \a}\eta_{\a,\b}^{~J,\g}\theta_\g^{~k}\wedge\theta_k^{~\b}\mod \phi, \theta_\a.$$
Since $\theta_\g^{~k}\wedge\theta_k^{~\b}$ has $(p-q)$-linearly independent $(1,0)$ forms, under the assumption that $(\ell'-q')<2(\ell-q)$, we obtain
$$\eta_{\a,\b}^{~J,\g}=0~\text{  if  }\g\neq \a$$
i.e.
$$\Theta_\a^{~J}=\eta_{\a,\b}^{~J}\phi_\a^{~\b},\quad 1\leq J\leq \ell-q~\text{or }2(\ell'-q')-(\ell-q)<J\leq p'-q'$$
for some $\eta_{\a,\b}^{~J}$ and \eqref{theta-null1-n} becomes
\begin{equation}\label{determine-eta}
\theta_\a^{~k}\wedge(\Omega_k^{~J}-\omega_k^{~J}-\eta_{\a,\b}^{~J}\theta_k^{~\b})
+\sum_{\ell-q<K\leq\ell'-q'}\Theta_\a^{~K}\wedge(\Omega_K^{~J}+\Omega_{2(\ell'-q')-K+1}^{~J})
=0\mod\phi.
\end{equation}

Choose a frame satisfying \eqref{theta1}, i.e.
$$\Theta_1^{~J}=0\mod \phi,\{\theta_\a,~\a\geq 2\}, \quad\ell-q<J\leq 2(\ell'-q')-(\ell-q)$$
and consider \eqref{determine-eta} for $\a=1.$
Then we obtain
$$\theta_1^{~k}\wedge(\Omega_k^{~J}-\omega_k^{~J}-\eta_{1,\b}^{~J}\theta_k^{~\b})=0\mod \phi, \theta_\a, \quad\a\geq 2,$$
which implies
$$\Omega_k^{~J}-\omega_k^{~J}=\eta_{1,\b}^{~J}\theta_k^{~\b}\mod\phi,\theta.$$
By substituting this into \eqref{determine-eta} for arbitrary $\a\leq q$, we obtain
$$(\eta_{1,\b}^{~J}-\eta_{\a,\b}^{~J})\theta_\a^{~k}\wedge\theta_k^{~\b}
+\sum_{\ell-q<K\leq\ell'-q'}\Theta_\a^{~K}\wedge(\Omega_K^{~J}+\Omega_{2(\ell'-q')-K+1}^{~J})
=0\mod\phi.$$
Since $\theta_\a^{~k}\wedge\theta_k^{~\b}$ has $(p-q)$-linearly independent $(1,0)$ forms, under the assumption that $(\ell'-q')<2(\ell-q),$ we obtain
$$\eta_{1,\b}^{~J}-\eta_{\a,\b}^{~J}=0,\quad \forall \a\leq q,$$
i.e.
$$\Theta_\a^{~J}=\eta_\b^{~J}\phi_\a^{~\b},\quad 1\leq J\leq \ell-q~\text{or }2(\ell'-q')-(\ell-q)<J\leq p'-q'$$
for some $\eta_\b^{~J}.$

By differentiating \eqref{theta-null2} using the structure equation and substituting \eqref{theta-null2}, we obtain
\begin{equation}\label{theta-null2-n}
\theta_a^{~k}\wedge(\Omega_k^{~J}-\Omega_k^{~2(\ell'-q')-J+1})+\sum_{\ell-q<K\leq\ell'-q'}\Theta_a^{~K}\wedge\3\Omega_K^{~J}=\3\eta_{a,\b}^{~J,\a}\theta_\a^{~j}\wedge
\theta_j^{~\b}\mod\phi
\end{equation}
for $J=\ell-q+1, \ldots, \ell'-q'$, where
$$\3\Omega_K^{~J}:=\Omega_K^{~J}+\Omega_{2(\ell'-q')-K+1}^{~J}-\left(\Omega_K^{~2(\ell'-q')-J+1}
+\Omega_{2(\ell'-q')-K+1}^{~2(\ell'-q')-J+1}\right),\quad J,K\leq \ell'-q'.$$
By counting the maximal linearly independent forms on the right and the left-hand sides of \eqref{theta-null2-n} as before, we obtain
$$\3\eta_{a,\b}^{~J,\a}=0\quad a>q,$$
i.e.
\begin{equation*}\label{theta-null2'}
\3\Theta_a^{~J}=0,\quad J=\ell-q+1, \ldots, \ell'-q',\quad a>q
\end{equation*}
and
$$\3 \eta_{\a,\b}^{~J,\g}=0,~\quad \a\neq \g,$$
$$\theta_\a^{~k}\wedge(\Omega_k^{~J}-\Omega_k^{~2(\ell'-q')-J+1}-\3\eta_{\a,\b}^{~J}\theta_k^{~\b})
+\sum_{\ell-q<K\leq\ell'-q'}\Theta_a^{~K}\wedge\3\Omega_K^{~J}
=0\mod\phi,~\quad \a\leq q$$
for some $\3 \eta_{\a,\b}^{~J}.$
Then by following the same argument, we obtain
$$\3\Theta_\a^{~J}=\3\eta_\b^{~J}\phi_\a^{~\b},\quad \ell-q< J\leq \ell'-q'$$
for some $\3 \eta_\b^{~J}.$

Now choose a frame change given by
$$
\widetilde Z_a'=Z_a',\quad \
\widetilde X_J'=X_{J}' + C_J^{~b}Z_b',\quad
\widetilde Y_a'=Y_a'+A_a^{~b}Z_b'+B_a^{~J}X_J',
$$
where
$$B_a^{~J}=\eta_a^{~J}\quad \text{   if  }a\leq q~\text{ and  }1\leq J\leq \ell-q~\text{or }2(\ell'-q')-(\ell-q)<J\leq p'-q' $$
and
$$B_a^{~J}=\3\eta_a^{~J}\quad\text{   if  }a\leq q~\text{ and  }\ell-q< J\leq \ell'-q' $$
and $0$ otherwise. Then the new $\Theta_a^{~J}$ satisfies
\begin{align*}
\Theta_a^{~J}&=0,\quad 1\leq J\leq \ell-q~\text{or }2(\ell'-q')-(\ell-q)<J\leq p'-q' ,\\
\3\Theta_a^{~J}&=0,\quad \ell-q< J\leq \ell'-q'.
\end{align*}

Similar argument for the case of $\Phi_1^{~1}=-\phi_1^{~1}$ will provide an analogous adaptation of the frame.
Summing up, we obtain the following proposition.
\begin{proposition}\label{pfaffian}
For any transversal CR-embedding $f\colon S^{\ell}_{q,p}\to S^{\ell'}_{q',p'}$
there is a smooth choice of $S^{\ell}_{q,p}$-frame and $S^{\ell'}_{q',p'}$-frame
such that the pulled back forms satisfy either
\begin{align}
\Phi_a^{~b}-\phi_a^{~b}&=0,\nonumber\\
\Theta_a^{~J}-\theta_a^{~J}&=0, ~1\leq J\leq \ell-q~\text{or }2(\ell'-q')-(\ell-q)<J\leq p'-q',\label{last-theta}\\
\3\Theta_a^{~J}&=0,~\ell-q<J\leq \ell'-q'\label{last-eq}
\end{align}
or
\begin{align}
\Phi_a^{~b}+\phi_a^{~b}&=0,\nonumber\\
\Theta_a^{~J}-\theta_a^{~p'-q'-J+1}&=0,~1\leq J\leq m-q\text{  or  }2(\ell'-q')-(m-q)<J\leq p'-q',\label{last-theta1}\\
\3\Theta_a^{~J}&=0,~m-q<J\leq \ell'-q'\label{last-eq-hat}
\end{align}
where
$$\3\Theta_a^{~J}=\Theta_a^{~J}-\Theta_a^{~2(\ell'-q')-J+1}.$$
\end{proposition}
\medskip

We call such a frame an $S^{\ell'}_{q',p'}$-\emph{frame adapted to} $f$.
If $\ell'-q'=\ell-q$ or $\ell'-q'=m-q$, then \eqref{last-eq} or \eqref{last-eq-hat} is an empty condition.
In \S \ref{null-sp}, we assume that $\ell'-q'>\ell-q$ or $\ell'-q'>m-q$ and then adapt $S^{\ell'}_{q',p'}$-frames further to $f$.
\medskip

%
After a frame change of $X_J'$ by rotation, we can choose a sequence of integers $n_1\leq\cdots\leq n_q\leq (\ell'-q')-(\ell-q)$ such that
\begin{equation}\label{hat-theta-indep}
\Theta_\a^{~\ell-q+n_{\a-1}+1}\wedge\cdots\wedge\Theta_\a^{~\ell-q+n_\a}\neq 0\mod\phi,
~\Theta_\a^{~\ell-q+n_\a+1}=\cdots=\Theta_\a^{~2(\ell'-q')-n_\a-1}=0\mod\phi.
\end{equation}
Then together with \eqref{last-theta}, we obtain
\begin{equation}\label{null-Z}
f_*(T^{1,0}_ZS_{q,p}^\ell)\subset Hom_{\mathbb{C}}(f(Z), (f(Z)+V_Z+W_Z)/F(Z)),
\end{equation}
where
$$V_Z:=\text{\rm span}\{X_j', ~j\leq \ell-q~\text{or }2(\ell'-q')-(\ell-q)<J\leq p'-q'\}$$
and
$$W_Z:=\text{\rm span}\{X_J'+X_{2(\ell'-q')-J+1}',~J=\ell-q+1,\ldots,\ell-q+n_q\}.$$


Finally by differentiating \eqref{last-theta} or \eqref{last-theta1} for $a>q$ and $J=j$,
we obtain either
\begin{equation}\label{Psi-a}
\Psi_a^{~\b}\wedge\theta_\b^{~j}+\sum_{\ell-q<K\leq \ell'-q'}\Theta_a^{~K}\wedge(\Omega_K^{~j}+\Omega^{~j}_{~2(\ell'-q')-K+1})=0,
\end{equation}
or
\begin{equation*}\label{Psi-a1}
\Psi_a^{~\b}\wedge\theta_\b^{~p'-q'-j+1}+\sum_{m-q<K\leq \ell'-q'}\Theta_a^{~K}\wedge(\Omega_K^{~j}+\Omega^{~j}_{~2(\ell'-q')-K+1})=0
\end{equation*}
for $1\leq j\leq \ell-q~\text{or }(p'-q')-(m-q)<j\leq p'-q'$.



\section{Fundamental forms of $f$ and fixed null space}\label{null-sp}
In this section, we assume $1<\ell-q$ and $\ell-q<\ell'-q'$ if $\Phi_1^{~1}=\phi_1^{~1}$ or $m-q<\ell'-q'$ if $\Phi_1^{~1}=-\phi_1^{~1}$.
Then we reduce the freedom of $S^{\ell'}_{q',p'}$-frames adapted to $f$.
We will show the reduction process under the assumption that $\Phi_1^{~1}=\phi_1^{~1}$.
The same argument will provide the same reduction of
freedom for the case of $\Phi_1^{~1}=-\phi_1^{~1}$.
\medskip

For a complex manifold $\mathcal{N}$ in a projective space, one can define projective fundamental forms $\mathbb{FF}^{(k)}_{\mathcal{N}},k\in \mathbb{N}$.(\cite{L})
We consider Grassmannian as a complex manifold in a projective space $\mathbb{P}^N$ under Pl\"{u}cker embedding.
Let $p\in \mathbb{P}^N$ and let $\mathcal{N}$ be a complex manifold in $\mathbb{P}^N$ through $p$. Then there exists $k_0$ depending on $p$ such that
$${\rm span }\{\mathbb{FF}^{(k)}_{\mathcal{N}}(p),k=1,\ldots,k_0\}={\rm span }\{\mathbb{FF}^{(k)}_{\mathcal{N}}(p),k\in \mathbb{N}\}.$$
Choose the smallest such $k_0$. At generic points of $\mathcal{N}$, $k_0$ is a constant.

Let $f:S^{\ell}_{q,p}\to S^{\ell'}_{q',p'}$ be a germ of a CR embedding. The Levi form of $S^{\ell}_{q,p}$ has $(\ell-q)$ negative and $(m-q)$ positive eigenvalues. Since we assumed $(m-q)\geq (\ell-q)\geq 1,$ the image of the Levi form becomes the complex normal bundle of $S_{q,p}^\ell.$ Hence $f$ extends to a neighborhood $U\subset Gr(q,p)$ of $S^{\ell}_{q,p}$ as a germ of a holomorphic embedding(\cite{BoP82}). A point $Z_0\in S^{\ell}_{q,p}$ is said to be {\em generic} if
$$ \dim{\rm span}\{\mathbb{FF}^{(k)}_{f(U)}(Z_0),k=1,\ldots,k_0\}=\max_{Z\in S^{\ell}_{q,p}}\left(\dim{\rm span }\{\mathbb{FF}^{(k)}_{f(U)}(Z),k=1,\ldots,k_0\}\right),\quad \forall k_0\geq 1.$$
\medskip

For a Grassmannian submanifold $Gr(q,F)\subset S^{\ell}_{q,p}$, denote by
$F_\sharp\subset \mathbb{C}^{p'+q'}$ the unique minimal subspace such that
$$f(Gr(q,F))\subset Gr(q',F_\sharp).$$
Note that since CR maps preserve complex submanifolds in CR manifolds, $Gr(q', F_\sharp)$ is a complex submanifold in $ S^{\ell'}_{q',p'}$ and
therefore
$$\langle~,\rangle_{\ell',m'}|_{F_\sharp}=0$$
and
$$F_\sharp\subset f(Z)+\text{\rm span}\{X_J',~1\leq J\leq p'-q'\}$$
for any $Z\in Gr(q, F)$ and $S^{\ell'}_{q',p'}$-frame $\{Z_a', X_J',Y_a'\}$ at $f(Z)$.
Let
$$\ell_\sharp:=\max(\dim F_\sharp),$$
where the maximum is taken over all Grassmannian submanifolds $Gr(q, F)\subset S^{\ell}_{q,p}$.
Note that by continuity of the fundamental forms, we can show that
$$\dim F_\sharp=\ell_\sharp$$
for generic maximal complex submanifolds $Gr(q, F)\subset S^{\ell}_{q,p}$.

Let
$$
   \mathcal{P}^\ell:=\{(Z,E)\in \mathcal{F}(q,\ell, p+q): Gr(q,E)\subset S_{q,p}^\ell\}
$$
and
$$
   \mathcal{P}'^{\ell_\sharp}:=\{(Z',E')\in \mathcal{F}(q',\ell_\sharp,p'+q'): Gr(q', E')\subset S_{q',p'}^{\ell'}\}.
$$
Then $\mathcal{P}^\ell$ and $\mathcal{P}'^{\ell_\sharp}$ are fiber bundles over $S_{q,p}^\ell$ and $S_{q',p'}^{\ell'}$, respectively.
As in \cite{Ng}, we define a bundle map $f^\sharp:\mathcal{P}^\ell\to \mathcal{P}'^{\ell_\sharp}$ by
$$ f^\sharp((Z,E))=(f(Z),E_\sharp).$$
Then by the property of projective fundamental forms, $f^\sharp$ is a meromorphic map on a neighborhood of $\mathcal{P}^\ell$. The reduction of the frames will be described in terms of the image of $f^\sharp$.
\medskip

From now on, we assume $\Phi_1^{~1}=\phi_1^{~1}$. Fix a generic point $Z\in S^{\ell}_{q,p}$ and choose an $S^{\ell'}_{q',p'}$-frame $\{Z_a', X_J', Y_a'\}$ adapted to $f$ at $f(Z)$. In particular, $f$ satisfies
\begin{align}
\Theta_a^{~J}-\theta_a^{~J}&=0,~1\leq J\leq \ell-q~\text{ or }~2(\ell'-q')-(\ell-q)<J\leq p'-q',\label{tangent} \\
\3\Theta_a^{~J}&=0,~\ell-q<J\leq \ell'-q'.\nonumber
\end{align}
Let
\begin{align*}
V_{Z}:&=\text{ span }\{X_j',~j\leq \ell-q~\text{or }j>(p'-q')-(m-q)\},\\
V_{Z}^\perp:&={\rm span}\{X_J',~\ell-q<J\leq (p'-q')-(m-q)\}.
\end{align*}
After a frame change by rotation, we may assume that $V_Z$ is orthogonal to $V_Z^\perp$ with respect to the standard Euclidean metric.
For $E\in \Lambda_{Z}$, we define $E_{V_{Z}}$ as follows:

After a suitable frame change by rotation, write
$$E=Z+\text{\rm span}\{\3 X_j,~j=1,\ldots,\ell-q\}$$
and let
$$E_{V_Z}:=\text{\rm span}\{ \3 X_j',~j=1,\ldots,\ell-q\}.$$

\begin{lemma}\label{null sp}
There exists a space $N_Z \in N(\ell',m')$
such that
$$N_Z\subset V^\perp_{Z}$$
and
\begin{equation}\label{in}
 E_\sharp\subset f(Z)+ E_{V_{Z}}+N_Z
\end{equation}
for all $E\in \Lambda_{Z}^\ell$.
\end{lemma}
Note that $N_Z$ depends on the choice of the frame $\{X_J'\}$.
The proof of Lemma~\ref{null sp} will be given in several steps.

Choose $E\in \Lambda_{Z}$ such that $\dim E_\sharp=\ell_\sharp$. Assume that
$$E=Z+\text{\rm span}\{\3 X_j,~j=1,\ldots,\ell-q\}.$$
Then
$Gr(q, E)$ is an integral manifold of
\begin{equation*}\label{defining}
\3\theta_\a^{~j}=0,\quad j=1,\ldots, \ell-q.
\end{equation*}
Since
$$\3\Theta_a^{~j}:=\Theta_a^{~j}-\Theta_a^{~p'-q'-j+1}=\3\theta_a^{~j},\quad j=1,\ldots,\ell-q,$$
we obtain
\begin{equation}\label{Theta-vanish}
\3\Theta_\a^{~j}=0,\quad j=1,\ldots, \ell-q
\end{equation}
with independence condition
\begin{equation}\label{independent'}
\bigwedge_{\a=1}^{q}(\Theta_\a^{~1}\wedge\cdots\wedge\Theta_\a^{~\ell-q})\neq 0.
\end{equation}
Since $Gr(q',E_\sharp)$ is a complex submanifold in $S_{q', p'}^{\ell'}$, $Gr(q',E_\sharp)$ is an integral manifold of
$$\Phi_a^{~b}=0,~\forall a, b.$$
Choose $b=\b\leq q.$ Then by structure equation for $\Phi$ with \eqref{tangent} and \eqref{Theta-vanish},
we obtain
$$d\Phi_a^{~\b}=\sum_{j\leq \ell-q}\3\Theta_a^{~j}\wedge\Theta_j^{~\b}=0.$$
Hence by \eqref{independent'} we obtain that on $Gr(q',E_\sharp)$,
\begin{equation}\label{tangent'}
\3\Theta_a^{~j}=0,\quad j=1,\ldots, \ell-q,
\end{equation}
which implies
%
%
that on $Gr(q',E_\sharp)$, we obtain
\begin{equation}\label{tangent''}
dZ_a'=\sum_{j\leq \ell-q}\Theta_a^{~j}\3X_j'+\sum_{\ell-q<J\leq p'-q'-(\ell-q)}\Theta_a^{~J}X_J'\mod Z'.
\end{equation}
By \eqref{theta1} of Lemma~\ref{r}, we may assume that
$$\Theta_1^{~J}=0\mod~\{\phi, \theta_\a,~\a\geq 2\},\quad \ell-q<J\leq 2(\ell'-q')-(\ell-q).$$
Consider
\begin{equation}\label{E'span}
dZ_1'=\sum_{j\leq \ell-q}\theta_1^{~j}\3X_j'\mod \theta_\a, \a\geq 2
\end{equation}
modulo $f(Z)$
on $f_*(T_{Z}Gr(q, E))$. Since $f_*(T_{Z}Gr(q, E))\subset T_{f(Z)} Gr(q',E_\sharp)$, we obtain
$$\3 X_j'\in E_\sharp,~j=1,\ldots,\ell-q,$$
i.e.
$$f(Z)+E_{V_Z}\subset E_\sharp.$$
In particular,
$$\ell_\sharp\geq q'+ (\ell-q).$$

Let $W_Z$ be as in \eqref{null-Z}. Then $W_Z$ is the smallest subspace in $N(\ell',m')$ such that
$$W_Z\subset V_Z^\perp$$
and
$$f_*(T_Z(S_{q,p}^\ell))\subset T_{f(Z)}(Gr(q, f(Z)+V_Z+W_Z)).$$
Moreover, since we assumed $\ell'-q'<2(\ell-q)$, by the independent condition \eqref{hat-theta-indep}, we may assume that the orthogonal projection of $E_\sharp$ to $W_Z$
is surjective.
Choose linearly independent vectors $L_J', ~J=\ell-q+1,\ldots,d_\sharp$ orthogonal to $N$ such that
$$E_\sharp=f(Z)+\text{\rm span}\{\3 X_j', L_J', ~j=1,\ldots,\ell-q,~J=\ell-q+1,\ldots,\ell-q+d_\sharp\}+ W_Z.$$
Since $Gr(q',E_\sharp)\subset S_{q',p'}^{\ell'}$, we obtain
$$\langle \3 X'_j, L_J'\rangle_{\ell',m'}=0,$$
which implies
$$L_J'\in \text{\rm span}\{\3 X_j', X_K',1\leq j\leq \ell-q,~\ell-q<K\leq p'-q'-(\ell-q)\}.$$
Hence we may assume that
$$L_J'\in \text{\rm span}\{ X_K',~\ell-q<K\leq p'-q'-(\ell-q)\}.$$
Let
\begin{equation*}\label{E-sharp-null}
L:=f(Z)+\text{\rm span}\{L_J',~j=\ell-q+1,\ldots, \ell-q+d_\sharp\}
\end{equation*}
so that
$$E_\sharp=\text{\rm span}\{\3 X_j', ~j=1,\ldots,\ell-q\}+L+ W_Z.$$
Note that
$$\text{\rm span}\{\3 X_j', ~j=1,\ldots,\ell-q\}=E_{V_Z}$$
and $L+W_Z$ is a null space with respect to $\langle~,\rangle_{\ell',m'}$.
Since $L_J'$, $J=\ell-q+1,\ldots,\ell-q+d_\sharp$ are linearly independent null vectors in $\text{\rm span}\{X_K',~\ell-q<K\leq p'-q'-(\ell-q)\}$, we obtain
$$d_\sharp\leq (\ell'-q')-(\ell-q).$$

Suppose
$d_\sharp>0.$
After a rotation of $\{X'_J\}_J$, we may assume that $L$ is orthogonal to $E_{V_Z}$ with respect to the standard Euclidean metric on $\mathbb{C}^{p'+q'}$.
Define
$$f_L:=\pi_{L}\circ f.$$
Here $\pi_L$ for a subspace $L\subset \mathbb{C}^{p'+q'}$ is defined to be a canonical map defined on a neighborhood of $f(Z)$ in $Gr(q', p')$ to $Gr(q', L)$ induced by orthogonal projection of $\mathbb{C}^{p'+q'}$ to $L$ with respect to the standard Euclidean metric.
%
In standard coordinates of $Gr(q', p')$ centered at $f(Z)$, we can write $f$ as a $q'\times p'$ matrix form
$f=(f_0, f_L),$
where
$$f_0:S^{\ell}_{q,p}\to Gr(q',f(Z)+L^\perp),~ f_L:S^{\ell}_{q,p}\to Gr(q',L).$$
For a Grassmannian submanifold $Gr(q,F)\subset S_{q,p}^\ell$ through $Z$, denote by $F_L$ the unique minimal subspace of $L$ such that
$$f_L(Gr(q,F))\subset Gr(q',F_L).$$

\begin{lemma}\label{sub-null}
Suppose that $d_\sharp> 0.$
Then there exists a subspace $F\subset E$ of dimension $q+1$ containing $Z$ such that
$$\dim F_L>q'.$$
\end{lemma}
\begin{proof}
We use local coordinates
$$\left\{\begin{pmatrix}
z_1^{1} & \cdots & z_{1}^{\ell-q}\cr
\vdots & \ddots & \vdots\cr
z_{q}^{1} & \cdots & z_{q}^{\ell-q }
\end{pmatrix}:z_{\a}^{j}\in \mathbb{C}
\right\}, \quad \left\{\begin{pmatrix}
w_1^{1} & \cdots & w_{1}^{d_\sharp}\cr
\vdots & \ddots & \vdots\cr
w_{q'}^{ 1} & \cdots & w_{q'}^{d_\sharp}
\end{pmatrix}:w_{a}^{J}\in \mathbb{C}\right\}$$
of $Gr(q,E)$ and $Gr(q',L)$ centered at $Z$ and $f(Z)$, respectively.
Write
$$f_L=(f_{a}),$$
where $f_{a}$ is a row vector for $a=1,\ldots,q'$.
Since $Z$ is a generic point, there exists $a$ such that the $k$-th fundamental form of $f_{a}$ is nontrivial at $0$ for some $k\geq 1$. In particular,
there exist $\a_i,j_i$, $i=1,\ldots,k$ such that
$$\frac{\partial^k f_{a}}{\partial z_{\a_1}^{j_1}\cdots \partial z_{\a_k}^{j_k}}(0)\neq 0.$$
Choose the smallest such $k$. Then a subspace $F$ defined by
$$F=Z+\mathbb{C}\left\{ c_1\3X_{j_1}+\cdots+c_k\3X_{j_k}\right\}$$
for some generic $c_1,\ldots,c_k\in \mathbb{C}$
satisfies the condition.
\end{proof}

Choose $F\subset E$ as in Lemma~\ref{sub-null}.
Since the fundamental forms depend smoothly on the point and direction, generic $ F\subset E$ containing $Z$ with $\dim F=q+1$ will satisfy
$$\dim F_{L}=\max(\dim \tilde F_L) ,$$
where the maximum is taken over all $ \tilde F\subset E$ containing $Z$ with $\dim \tilde F=q+1.$
Now assume that there exists a sequence of subspaces
$Z\subset F_1\subset\cdots\subset F_d\subset E$ of $\dim F_j=q+j$ such that
$\dim F_{j-1,L}<\dim F_{j,L}$ and
$$\dim F_{j,L}=\max(\dim \tilde F_{j,L}),$$
where the maximum is taken over all $\tilde F_j\subset E$ containing $F_{j-1}$ with $\dim \tilde F_j=q+j.$
Suppose $F_{d,L}\neq L.$
Choose $F_{d,L}^\perp \subset L$ and define
$$\tilde f:=\pi_{F_{d,L}^\perp} f_{L}.$$
Since $F_{d,L}\neq L,$ $\tilde f$ is nontrivial.
In standard coordinates, write
$$\tilde f=(\tilde f_{ a}),$$
where $\tilde f_{a}$ is a row vector for $a=1,\ldots,q'$.
Since $Z$ is a generic point, there exists $a$ such that the $k$-th fundamental form of $\tilde f_{a}$ at $Z$ is nontrivial for some $k\geq 1$.
Then by the same argument of Lemma \ref{sub-null}, we can choose $ F_{d+1}\subset  E$ containing $ F_{d}$ such that
$$\dim F_{d+1,L}>\dim F_{d,L}.$$

We repeat the same procedure to choose a sequence of subspaces
$Z\subset F_1\subset\cdots\subset F_d\subset E$ with $\dim F_j=q+j$ such that
$\dim F_{j-1,L}<\dim F_{j,L}.$ Since $\ell'-q'<2(\ell-q)$ and $d_\sharp\leq \ell'-q'-(\ell-q)$, this sequence terminates for some $d<\ell-q.$ Therefore there exists an $(\ell-1)$-dimensional subspace $F$ of $E$ containing $Z$ such that
$F_L=L.$
Furthermore, by continuity of the fundamental forms, we can choose an open set $\mathcal{U}_E$ in $Gr(\ell-1,E)$ such that
$ F_L=L$ for all $ F\in \mathcal{U}_E$.
\medskip

{\em proof of Lemma~\ref{null sp}} :
Let $E$ and $L$ be as before.
Choose $\tilde E\in \Lambda_{Z}^\ell$ such that
$\dim E\cap \tilde E=\ell-1.$
We may assume
$$\tilde E=Z+\text{\rm span}\{T, \3 X_j, j=2,\ldots,\ell-q\}$$
for some null vector $T$ transversal to $E$.
Let
$$T=c^k X_k$$
and let
$$T'=\sum_{k\leq \ell-q}c^k X_k'+\sum_{k>p'-q'-(m-q)}c^k X_k'.$$
By considering \eqref{E'span} on $f_*(T_{Z}Gr(q, \tilde E))$, we obtain
$$\tilde E_{V_Z}=\text{\rm span}\{T', \3 X_j',j=2,\ldots,\ell-q\}\subset \tilde E_\sharp.$$
Assume that
$F:=E\cap\tilde E\in \mathcal{U}_E$.
Then we obtain
$$L=F_L=\tilde E_L$$
and therefore
$$\pi_L(\tilde E_\sharp)=L,$$
where $\pi_L$ is the orthogonal projection to $L$.
Since $Gr(q', \tilde E_\sharp)\subset S_{q',p'}^{\ell'}$, we obtain
\begin{equation*}
\langle T', L\rangle_{\ell',m'}=0.
\end{equation*}
In particular,
\begin{equation}\label{orthog}
\langle T', L_J'\rangle_{\ell',m'}=0,\quad J=\ell-q+1,\ldots,\ell-q+d_\sharp.
\end{equation}

Since $\tilde E$ is arbitrary, \eqref{orthog} implies that
$L_J'$ is orthogonal to $\text{\rm span}\{X_j',~j\leq \ell-q~\text{  or  }~j>p'-q'-(m-q)\}$,
i.e.
$$L_J'\in V_Z^\perp,\quad J=\ell-q+1,\ldots,\ell-q+d_\sharp.$$

%
%

Let
$$\mathcal{U}_0=\{E\}$$
and let
$$\mathcal{U}_j:=\bigcup_{F\in \mathcal{U}_{j-1}}\{\tilde E\in \Lambda_{Z}^\ell:{\rm dim}~ \tilde E\cap F\geq \ell-1 \}, \quad j\in \mathbb{N}.$$
Then by induction argument on $j$, we can show that for each $j$, generic $ \tilde E\in \mathcal{U}_j$ satisfies
$$ \pi_L(\tilde E_\sharp)=L.$$
Since
$$\tilde E_{V_Z}\subset \tilde E_\sharp\cap V_Z$$
and $L$ is orthogonal to $V_Z$, we obtain
$$\tilde E_\sharp\supset \tilde E_{V_Z}+L.$$

Let
$$N_Z:=W_Z+\text{\rm span}\{L_J',~J=\ell-q+1,\ldots,\ell-q+d_\sharp\}.$$
By counting dimension, we obtain
$$\tilde E_\sharp\subset f(Z)+\tilde E_{V_Z}+N_Z.$$
Then by Lemma~\ref{open}, we can show that \eqref{in} holds
for all $\tilde E\in \Lambda_{Z}^\ell$, which completes the proof.
\medskip

Next we will show the following.
\begin{lemma}\label{frame choice}
There exist subspace $N_\sharp\in N(\ell',m')$ and a choice of smooth $S^{\ell'}_{q',p'}$-frame adapted to $f$ such that
$$N_\sharp\subset V_Z^\perp,\quad \forall Z$$
and
$$ E_\sharp\subset f(Z)+E_{V_Z}+N_\sharp$$
for all $E\in \Lambda_Z^\ell$.
\end{lemma}
\begin{proof}
Fix a generic point $Z_0\in S_{q,p}^\ell$ and an $S_{q',p'}^{\ell'}$-frame adapted to $f$ at $f(Z_0)$.
Let
$$N_\sharp:=N_{Z_0}\subset V^\perp_{Z_0}$$
be as in Lemma~\ref{null sp}
and let $Gr(q, F)$ be a submanifold of $S^{\ell}_{q,p}$ containing $Z_0$ such that $\dim F=q+1$. Assume
that
$$F=Z_0+\mathbb{C}\3 X_1.$$
Choose a generic point $Z_1\in Gr(q, F)$. Since $Z_0\in Gr(q, F)$, we obtain
$$\pi_{N_\sharp}( E_\sharp)=N_\sharp.$$
for all generic $  E\in \Lambda_F.$
Then by the same argument as for $Z_0$, we can show that
$$\pi_{N_\sharp}(W_\sharp)=N_\sharp$$
for all generic $ W\in \Lambda_{Z_1}.$

Choose vectors $T$ and $T'$ such that
$$F=Z_0+\mathbb{C}\3 X_1=Z_1+\mathbb{C}T.$$
By Lemma~\ref{null sp}, we obtain
$$F_\sharp\subset f(Z_0)+\3X_1'+N_\sharp$$
and hence
there exists $T'$ such that
$$f(Z_0)+\mathbb{C}\3 X_1'=f(Z_1)+\mathbb{C}T'\mod N_\sharp.$$
Let $W\in \Lambda_F^\ell$. We may assume that
$$W=F+\text{\rm span}\{\3 X_j,~j=2,\ldots,\ell-q\}.$$
Then in view of \eqref{tangent''}, we obtain
$$d Z'_a=\Theta_a^{~1}T'+\sum_{j=2}^{\ell-q}\Theta_a^{~j}\3 X_j'\mod f(Z_1)+N_\sharp$$
and therefore
$$W_\sharp\subset f(Z_1)+\text{\rm span}\{T', \3 X_j', ~2\leq j\leq \ell-q\} +N_\sharp.$$
Since $W$ is arbitrary, we obtain
$$W_\sharp\subset f(Z_1)+\text{\rm span}\{T', X_j', ~2\leq j\leq \ell-q\text{   or   }j>p'-q'-(m-q)\} +N_\sharp.$$
for all $W\in \Lambda_{F}^\ell.$ Since $\dim F=q+1,$ there exists a null vector $\tilde T$ such that
$\{T+\tilde T, T-\tilde T, X_j,~j=2,\ldots,\ell-q\}$ spans the complex tangent space of $S_{q,p}^\ell$ at $Z_1$.
Then by continuity of the fundamental forms together with \eqref{tangent''}, we can choose a null vector $\tilde T'$ orthogonal to
$N_\sharp$ such that
$$W_\sharp\subset f(Z_1)+\text{\rm span}\{T',\tilde T', X_j', ~2\leq j\leq \ell-q\text{   or   }j>p'-q'-(m-q)\} +N_\sharp$$
for any $W\in \Lambda_{Z_1}^\ell$ containing $Z_1+\mathbb{C}\tilde T$.
Let
$$V_{Z_1}=\text{\rm span}\{T'+\tilde T',T'-\tilde T', X_j', ~2\leq j\leq\ell-q\text{   or   }j>p'-q'-(m-q) \}.$$
Then by the same argument, we can show that
$$E_\sharp\subset f(Z_1)+E_{V_{Z_1}}+N_\sharp,\quad \forall E\in \Lambda_{Z_1}^\ell.$$

By Lemma~\ref{flat}, every point in $S_{q,p}^\ell$ is connected with $Z_0$ by a chain consisting of the form $Gr(q, F_k)$, $k=1,2,\ldots,K$ for some $K$. Moreover, we can choose a chain such that $\dim F_k=q+1$ for all $k$.
Hence by iterating this process along chains,
and applying Lemma~\ref{chain}, we can show that
for each $Z\in S_{q,p}^\ell$, there exists a choice of $S_{q',p'}^{\ell'}$-frame adapted to $f$
that satisfies the conditions in the lemma.

%
%


Now consider a linear subspace $L_{Z_1}$ spanned by $\{E_\sharp:E\in \Lambda_{Z_1}^\ell\}.$
Since $Gr(q', E_\sharp)\subset S^{\ell'}_{q',p'}$ for all $E\in \Lambda_Z^\ell$, we obtain
$$L_{Z_1}\subset \text{\rm span}\{Z'_a, X'_J\},$$
where $\{Z_a',X_J', Y_a'\}$ is an $S_{q',p'}^{\ell'}$-frame at $f(Z_1)$.
Since $f$ is smooth, $\dim L_{Z_1}$ is constant on an open set of $S_{q,p}^\ell$ and therefore $\bigcup_{Z\in S^{\ell}_{q,p}}L_{Z}$ is a smooth vector bundle over an open set of $S^{\ell}_{q,p}$.
Furthermore, since $Gr(q', E_\sharp)\subset S^{\ell'}_{q',p'}$ for all $E\in \Lambda_Z$, we obtain
$$L_{Z}\subset \text{\rm span}\{Z'_a, X'_J\},$$
where $\{Z_a',X_J', Y_a'\}$ is an $S_{q',p'}^{\ell'}$-frame at $f(Z)$ satisfying the conditions in the lemma.

Since
$$N_\sharp\subset L_{Z},~\forall Z\in S^{\ell}_{q,p},$$
we can choose a smoothly varying $S^{\ell'}_{q',p'}$-frame adapted to $f$ such that
$$L_{Z}=f(Z)+\text{\rm span}\{X_1',\ldots,X_{\ell-q}',X_{p'-q'-(m-q)+1}',\ldots,X_{p'-q'}'\}+N_\sharp,$$
which completes the proof.
\end{proof}

\section{partial rigidity and the proof of Theorems}\label{s-rigid}
In this section, we will prove Theorem~\ref{main} and Corollary~\ref{main-cor} under the assumption that $\Phi_1^{~1}=\phi_1^{~1}.$ The same argument works for the case of $\Phi_1^{~1}=-\phi_1^{~1}$ to complete the proof.
\medskip

Let $N_\sharp$ be a subspace as in Lemma~\ref{frame choice}.
\begin{lemma}\label{Za}
There exists an $S^{\ell'}_{q',p'}$-frame adapted to $f$ such that
$$dZ_a'=0\mod \text{\rm span} \{Z_a',~a>q\}+N_\sharp,~\forall a>q.$$
\end{lemma}
\begin{proof}
Let $\{Z_a', X_J', Y_a'\}$ be an $S^{\ell'}_{q',p'}$-frame adapted to $f$ satisfying the condition in Lemma~\ref{frame choice} such that
\begin{equation}\label{vanish-0}
N_\sharp={\rm span }\{\3 X_J', ~J=\ell-q+1, \ldots,\ell-q+n_\sharp\},
\end{equation}
where
$$n_\sharp=\dim N_\sharp$$
and
$$\3 X_K'=X_K'+X_{2(\ell'-q')-K+1}',~K=\ell-q+1, \ldots,\ell-q+n_\sharp.$$
Since $N_\sharp$ is a fixed subspace, we obtain
$$d\3 X_K'=d X_K'+dX_{2(\ell'-q')-K+1}'=0\mod N_\sharp,\quad K=\ell-q+1, \ldots,\ell-q+n_\sharp,$$
which in view of \eqref{differential} implies
\begin{equation}\label{vanish-0}
\Omega_K^{~j}+\Omega^{~j}_{~2(\ell'-q')-K+1}=0,~ j=1,\ldots,\ell-q,~ K=\ell-q+1,\ldots,\ell-q+n_\sharp.
\end{equation}
Moreover, since
$$f(Gr(q, E))\subset Gr(q',f(Z)+E_{V_Z}+N_\sharp)$$
with
$$E_{V_Z}\subset V_Z=\text{\rm span}\{X_1',\ldots,X_{\ell-q}',X_{p'-q'-(m-q)+1}',\ldots,X_{p'-q'}'\},$$
we obtain
\begin{equation}\label{vanish-2}
\Theta_a^{~J}=0, ~J=\ell-q+n_\sharp+1,\ldots,2(\ell'-q')-(\ell-q)-n_\sharp.
\end{equation}

Let $a>q.$
Then \eqref{Psi-a} becomes
$$\Psi_a^{~\b}\wedge\theta_\b^{~j}=0.$$
Since $\Psi_a^{~\b}$ is independent of $j$ and $p-q>1$, we obtain
$$\Psi_a^{~\b}=0,\quad \forall \b.$$
Therefore by Proposition~\ref{pfaffian} and \eqref{vanish-2},
we obtain
$$dZ_a'=\sum_{\ell-q<J\leq \ell-q+n_\sharp } \Theta_a^{~J}\3 X_J'~\mod  \text{\rm span} \{Z_a',~a>q\},$$
which completes the proof.
\end{proof}

By Lemma~\ref{Za}
we obtain
$$dZ_a'=0~\mod  \text{\rm span} \{Z_a',~a>q\}+N_\sharp.$$
Since $N_\sharp$ is fixed,
${\rm span}\{Z_a',~a>q\}+N_\sharp$ is also fixed.
Fix a generic point $Z_0\in S^{\ell}_{q,p}$. We may assume that
$$Z_0=\text{\rm span}\{\3 e_\a, ~\a=1,\ldots,q\},\quad f(Z_0)=\text{\rm span}\{\3 e_a', ~a=1,\ldots,q'\}$$
for
$$\3 e_\a:=e_\a+e_{p+q-\a+1},~\a=1,\ldots,q,\quad \3 e_a':=e'_a+e'_{p'+q'-a+1},~a=1,\ldots,q'$$
where $e_1,\ldots,e_{p+q}$ and $e'_1,\ldots,e'_{p'+q'}$ are standard basis of $\mathbb{C}^{p+q}$ and $\mathbb{C}^{p'+q'}$,
respectively.
Assume further that
\begin{equation}\label{N sharp}
N_\sharp= \text{\rm span}\{\3e'_{q'+\ell-q+1},\ldots,\3e'_{q'+\ell-q+n_\sharp}\}
\end{equation}
and
\begin{equation}\label{fixed space}
{\rm span}\{Z_a',~a>q\}+N_\sharp= \text{\rm span}\{\3e'_{q+1},\ldots,\3e'_{q'}\}+\text{\rm span}\{\3e'_{q'+\ell-q+1},\ldots,\3e'_{q'+\ell-q+n_\sharp}\},
\end{equation}
where
$$\3 e_{q'+\ell-q+J}':=e'_{q'+J}+e'_{q'+2(\ell'-q')-J+1},~J=1,\ldots,n_\sharp.$$
Since $ V_Z$ is orthogonal to $f(Z)+N_\sharp$ with respect to $\langle~,\rangle_{\ell',m'}$,
after suitable frame changes by rotation, we may assume that
$V_Z$ is orthogonal to $\text{\rm span}\{\3 e'_a,~a>q\}+N_\sharp$ with respect to the standard Euclidean metric.

In standard coordinates of $Gr(q', p')$ centered at $f(Z_0)$, we will write $f$ as a $q'\times p'$ matrix form
$f=(f_0, f_N),$
where
\begin{equation}\label{fix Z0}
f_0:S^{\ell}_{q,p}\to Gr(q',N_\sharp^\perp),~ f_N:S^{\ell}_{q,p}\to Gr(q',f(Z_0)+N_\sharp).
\end{equation}
Note that
$$f(Z)=f_0(Z)\mod N_\sharp.$$

\begin{lemma}\label{flat}
For each $Gr(q, E)\subset S^{\ell}_{q,p}$ with $\dim E=\ell$, there exists a subspace $ E'\subset N_\sharp^\perp$ of dimension $\ell$ orthogonal to $\text{\rm span}\{\3 e_{q+1}',\ldots,\3 e_{q'}'\}$ such that
$$f_0(Gr(q,E))=Gr(q , E')\oplus \text{\rm span}\{\3e'_{q+1},\ldots,\3e'_{q'}\}.$$
Furthermore, if $\ell-q>1,$ then $f_0$ modulo $\text{\rm span}\{\3e'_{q+1},\ldots,\3e'_{q'}\}$ is a standard isomorphism.
\end{lemma}
\begin{proof}
Choose an $S_{q',p'}^{\ell'}$-frame adapted to $f$ satisfying the condition in Lemma~\ref{Za}.
Then we obtain
$$\text{\rm span}\{Z_a',~ a>q\}=\text{\rm span}\{\3 e'_{q+1},\ldots,\3 e'_{q'}\}\mod N_\sharp,$$
which implies
$$f_0(Z)=\text{\rm span}\{ Z_\a'',~ \a=1,\ldots,q\}\oplus\text{\rm span}\{\3 e'_{q+1},\ldots,\3 e'_{q'}\},$$
where $ Z_\a''$ is the orthogonal projection of $Z_\a'$ to $N_\sharp^\perp$ with respect to Euclidean inner product.
Let $Gr(q, E)\subset S^{\ell}_{q,p}$ be a maximal complex submanifold and let $ E'\subset N_\sharp^\perp$ be the smallest subspace orthogonal to $\text{\rm span}\{\3 e_{q+1}',\ldots,\3 e_{q'}'\}$ such that
$$f_0(Gr(q, E))\subset Gr(q, E')\oplus \text{\rm span}\{\3 e'_{q+1},\ldots,\3 e'_{q'}\}.$$
We will show that $ E'$ is of dimension $\ell$ and $f_0$ is a projective linear isomorphism if $\ell-q>1$.

If $\ell-q=1$, then \eqref{main-ineq} implies that $\ell'-q'=1$. Therefore
$$\dim N_\sharp=0$$
and
$$f_0(Z)=f(Z)\in \{Z'\in S^{\ell'}_{q',p'}: Z'_a=\3 e_a',~a>q\}.$$
Since maximal complex submanifolds in $S_{q',p'}^{\ell'}$ are of the form $Gr(q', L)$ with $\dim L=\ell'$, maximal complex submanifolds in $\{Z'\in S^{\ell'}_{q',p'}: Z'_a=\3 e_a',~a>q\}$ are of the form
$Gr(q, E')\oplus \text{\rm span}\{\3 e'_{q+1},\ldots,\3 e'_{q'}\}$ with
$$\dim E'=\ell'-(q'-q)=(\ell-q)+q=\ell.$$

Assume that $\ell-q>1.$ Choose a generic point $Z\in Gr(q, E)$.
By Lemma~\ref{frame choice}, $E_\sharp=f(Z)+E_{V_{Z}}+N_\sharp$ is the smallest subspace such that
$$f(Gr(q, E))\subset Gr(q',E_\sharp).$$
Since
$$E_{V_Z}\subset V_Z\subset N_\sharp^\perp,$$
the space
$$E':=\text{\rm span}\{Z_\a'',~ \a=1,\ldots,q\}+E_{V_{Z}}$$
is the smallest subspace such that
$$f_0(Gr(q, E))\subset Gr(q,  E')\oplus \text{\rm span}\{\3 e'_{q+1},\ldots,\3 e'_{q'}\}.$$
Hence we obtain
$$\dim  E'=q+\dim E_{V_Z}=\ell.$$

Let $Z$ be an arbitrary generic point of $Gr(q, E)$. Since
$$ f(Gr(q, F))\subset Gr(q',f(Z)+ F_{V_{Z}}+N_\sharp),~\forall F\in \Lambda_{Z},$$
we obtain
$$ f(Gr(q, F_1\cap F_2))\subset Gr(q',f(Z)+ F_{1,V_{Z}}+N_\sharp)\cap Gr(q',f(Z)+ F_{2,V_{Z}}+N_\sharp)$$
$$=Gr(q',f(Z)+ (F_{1,V_{Z}}\cap F_{2,V_{Z}})+N_\sharp),~\forall F_1, F_2\in \Lambda_{Z}.$$
Choose generic $F_1,\ldots,F_{\ell-q-1}\in \Lambda_{Z}$ such that
$\cap_j F_j\subset E$ and $\dim \cap_j F_j$ is $q+1.$ By induction, we obtain
$$ f(Gr(q, \bigcap_j F_j))\subset Gr(q', f(Z)+\bigcap_j(F_{j,V_{Z}})+N_\sharp).$$
Therefore
$$f_0(Gr(q, \bigcap_j F_j))\subset Gr(q', f_0(Z)+\bigcap_j(F_{j,V_{Z}})).$$
Since
$$\dim \bigcap_j (F_{j,V_{Z}})=\dim \left(\bigcap_j F_j\right)_{V_{Z}}=1,$$
$f_0$ preserves the variety of minimal rational tangents(See \cite{HwM99} for definition).
Since the rank of $Gr(q, E)$ is equal to $\min(\ell-q, q)$ which is strictly bigger than $1$ by our assumption, the rank of $Gr(q ,E')$ is also strictly bigger than $1$. Then by \cite{M08}, we can complete the proof.
\end{proof}

\begin{lemma}
There exists a $(p+q)$-dimensional subspace $L\subset N_\sharp^\perp$ orthogonal to $\text{\rm span}\{\3 e'_{q+1},\ldots,\3 e'_{q'}\}$ such that
$$\langle ~,\rangle_{\ell',m'}|_L=\langle~, \rangle_{\ell,m}$$
and
$$f_0(S^{\ell}_{q,p})\subset Gr(q, L)\oplus\text{\rm span}\{\3 e'_{q+1},\ldots,\3 e'_{q'}\}.$$
\end{lemma}
\begin{proof}
Let $Z_0$, $f(Z_0)$ and $N_\sharp$ be fixed as before. Assume further that
$$V_{Z_0}=\text{\rm span}\{e'_{q'+j},~1\leq j\leq \ell-q~\text{or }(p'-q')-(m-q)<j\leq p'-q'\}$$
and
$$Y_{Z_0}'=\text{\rm span}\{\Check e'_a,~a=1,\ldots,q'\},$$
where
$$\Check e'_a:=e'_a-e'_{p'-q'-a+1},~a=1,\ldots,q'.$$
Let
$$L:=\text{\rm span}\{\3e'_\a,\Check e_\a',~\a=1,\ldots,q\}+V_{Z_0}.$$
Then $L$ is a $(p+q)$-dimensional subspace in $N_\sharp^\perp$ such that
$$\langle ~,\rangle_{\ell',m'}|_L=\langle~, \rangle_{\ell,m}.$$

Choose an $S_{q',p'}^{\ell'}$-frame $\{Z'_a, X_J', Y_a'\}$ adapted to $f$ that satisfies the conditions in Lemma~\ref{frame choice} and Lemma~\ref{Za}.
Define
\begin{align*}
Z'':&=\text{\rm span}\{Z''_\a,~\a=1,\ldots,q\},\\
X'_{Z}:&=\text{\rm span}\{X_J',~1\leq J\leq p'-q'\},\\
Y'_{Z}:&=\text{\rm span}\{Y_a',~a=1,\ldots, q'\}\\
Y_{Z}'':&={\rm span}\{Y''_\a, ~\a=1,\ldots,q\},
\end{align*}
where $Z_\a''$ and $Y''_\a$ are the orthogonal projections of $Z_\a'$ and $Y'_\a$ to $N_\sharp^\perp,$ respectively.
Note that
$$f_0(Z)=Z''\oplus\text{\rm span}\{\3e'_{q+1},\ldots,\3 e'_{q'}\} .$$
Choose an orthogonal complement
$X_{Z}''\subset X_Z'$ of $N_\sharp$ with respect to $\langle~, \rangle_{\ell',m'}$. After a frame change by rotation, we may assume $X_Z''$ is orthogonal to $\text{\rm span}\{\3 e'_{q+1},\ldots,\3 e_{q'}'\}+N_\sharp$ with respect to the standard Euclidean metric on $\mathbb{C}^{p'+q'}$ and that
\begin{equation}\label{VZ}
V_Z\subset X_Z''
,\quad\forall Z\in S^{\ell}_{q,p}.
\end{equation}
Since
$Z''+X''_Z+Y''_Z$ is an orthogonal complement of $\text{\rm span}\{\3 e_a',~q<a\leq q'\}$ in $N_\sharp^\perp$ with respect to $\langle~, \rangle_{\ell',m'}$, $f_0(Z)+X''_Z+Y''_Z$ is well-defined independently of the point $Z\in S_{q,p}^\ell$ and frames.

Since $S_{q',p'}^{\ell'}$-frame satisfies
$$dZ_a'=\Theta_a^{~J}X_J'+\Phi_a^{~b}Y_b'\mod Z',$$
Proposition~\ref{pfaffian} implies
\begin{equation}\label{image}
T^{1,0}_{f_0(Z)}f_0(S_{q,p}^\ell)\subset Hom_\mathbb{C}(Z'', (Z''+V_Z)/Z'')
\end{equation}
and
$$T_{f_0(Z)}f_0(S_{q,p}^\ell)\subset Hom_\mathbb{R}(Z'', (Z''+V_Z+Y_Z'')/Z'').$$
We claim that
$$f_0(Z)+V_Z+Y_Z''= L\oplus\text{\rm span}\{\3e'_{q+1},\ldots,\3 e'_{q'}\},\quad \forall Z\in S^{\ell}_{q,p}.$$
We will prove the claim for all $Z\in Gr(q, F)\subset S^{\ell}_{q,p}$, where $F$ satisfies $Z_0\subset F$ and $\dim F=q+1$. By Lemma~\ref{flat}, every point in $S_{q,p}^\ell$ is connected with $Z_0$ by a chain consisting of the form $Gr(q, F_k)$, $k=1,2,\ldots,K$ for some $K$. Moreover, we can choose a chain such that $\dim F_k=q+1$ for all $k$.
Hence by iterating this process along chains, we can complete the proof.




Let $Gr(q, F)$ be a submanifold of $S^{\ell}_{q,p}$ containing $Z_0$ such that $\dim F=q+1$. Assume
$$F=Z_0+\mathbb{C}\3e_{q+1},$$
where
$$\3 e_{q+1}:=e_{q+1}+e_{p}.$$
By applying Lemma~\ref{flat}, 
$$f_0(Gr(q,F))=Gr(q, F')\oplus \text{\rm span}\{\3 e_\a, ~q<a\leq q'\},$$
where
$$ F'=\text{\rm span}\{\3 e_\a', \3 e'_{q'+1},~\a=1,\ldots,q\}.$$
Fix a point $Z_1\in Gr(q, F)$ and assume that
$$f_0(Z_1)=\text{\rm span}\{\3e_1'+\3 e'_{q'+1},\3e_\a',~\a=2,\ldots,q\}\oplus \text{\rm span}\{\3 e_a',~ q<a\leq q'\}.$$
To prove the claim, it is enough to show that
$$\text{\rm span}\{\3e_1'+\3e'_{q'+1},\3e_\a',~\a=2,\ldots,q\}+V_{Z_1}+Y_{Z_1}''=L\mod \text{\rm span}\{\3e_a', ~q<a\leq q'\}.$$

Consider $f^\sharp$ on $\mathcal{P}_{Z_0}^\ell$. Then Lemma~\ref{flat} implies
$$f^\sharp(\mathcal{P}_{Z_0}^\ell)\subset f_0(Z_0)\oplus Gr(\ell-q, V_{Z_0})\oplus \text{\rm span}\{\3 e_a', ~a>q\}\oplus N_\sharp.$$
Since $\mathcal{P}_{Z_0}^\ell$ is the Shilov boundary of $D_{Z_0}$, where we define
$$D_{Z}:=\{E\in Gr(\ell, m):Z\subset E, ~\langle E, E\rangle_{\ell,m}\geq 0\},$$
we obtain
$$f^\sharp(D_{Z_0})\subset f_0(Z_0)\oplus Gr(\ell-q, V_{Z_0})\oplus \text{\rm span}\{\3 e_a', ~a>q\}\oplus N_\sharp.$$
Similarly, we obtain
$$f^\sharp(D_{Z_1})\subset f_0(Z_1)\oplus Gr(\ell-q, V_{Z_1})\oplus \text{\rm span}\{\3 e_a', ~a>q\}\oplus N_\sharp$$
and hence
$$f^\sharp(D_{Z_0}\cap D_{Z_1})\subset \left(f_0(Z_0)\cap f_0(Z_1)\right)\oplus Gr(\ell-q, V_{Z_1}\cap V_{Z_1})\oplus \text{\rm span}\{\3 e_a', ~a>q\}\oplus N_\sharp.$$
Since
$$D_{Z_0}\cap D_{Z_1}=F+\text{\rm span}\{e_j,~j=2,\ldots, p-q-1\},$$
we obtain
\begin{equation}\label{V-Z1}
V_{Z_1}\supset F'+ \text{\rm span}\{e_j',~1<j\leq \ell-q \text{  or  }(p'-q')-(m-q)< j<p'-q'\}.
\end{equation}

For an $S^{\ell'}_{q',p'}$-frame $\{Z_a', X_J', Y_a'\}$ adapted to $f$ at $f(Z_1)$, assume
$$X_1'+X_{p'-q'}'=\3e'_{q'+1}$$
so that
$$ F'=\text{\rm span}\{\3e_1'+\3e'_{q'+1},~\3e_\a',~2\leq \a\leq q\}+\mathbb{C}(X_1'+X_{p'-q'}').$$
Since $\{Z'_a, X_J', Y_a'\}$ is an $S^{\ell'}_{q',p'}$-frame at $f(Z_1)$, we obtain
$$\langle Z'_a,X'_J\rangle_{\ell',m'}=0,$$
i.e.
$$\langle \3e_1'+\3e'_{q'+1},X_{J}'\rangle_{\ell',m'}=\langle \3e_a',X_J'\rangle_{\ell',m'}=0,~a=2,\ldots,q'.$$
Hence we may assume that
$$X_1'=e'_{q'+1}+\sum_{q'+2\leq J<p'}b_Je'_{J}+\Check e_{1}',
\quad X_{p'-q'}'=e'_{p'}-\sum_{q'+2\leq J<p'}b_Je'_{J}-\Check e_{1}'
$$
modulo $f(Z_1)$ for some $b_J\in \mathbb{C}$.
Since
$$\langle X_1',X_{p'-q'}'\rangle_{\ell',m'}=0,$$
we obtain
\begin{equation}\label{b_J}
\sum_{q'+2\leq J\leq \ell'}|b_J|^2-\sum_{\ell'<J<p'}|b_J|^2=0.
\end{equation}
By \eqref{V-Z1}, we may assume
$$X_1'=e_{q'+1}'+T+\check e_{q'+1}',\quad X_{p'-q'}'=e'_{p'}-T-\check e_{q'+1}'\mod f(Z_1)$$
for some vector $T$ such that
$$\langle T, T\rangle_{\ell',m'}=\langle T, e'_{q'+j}\rangle_{\ell',m'}=0,~1<j\leq\ell-q~\text{or }(p'-q')-(m-q)< j<p'-q'.$$

We will show that
$$V_{Z_1}=\text{\rm span}\{e'_{q'+1}+\Check e'_1,e'_{p'}-\Check e'_1, e'_{q'+j},~1<j\leq \ell-q \text{  or  }(p'-q')-(m-q)< j<p'-q'\}\mod f(Z_1).$$
If $\ell'-q'=1$, then 
\eqref{b_J} becomes
$$-\sum_{\ell'<J<p'}|b_J|^2=0.$$
Therefore the conclusion follows.
Assume that $\ell'-q'>1.$
Consider a set $\{V_{s,t}:s,t\in \mathbb{C}\}\subset Gr(q', N_\sharp^\perp)$
defined by
$$V_{s,t}:=\text{\rm span}\{\3e_1'+s(\3e'_{q'+1}-\sqrt{-1}t \Check e'_{q'+1}),\3e_a',~a=2,\ldots,q'\}$$
with
$$\Check e'_{q'+1}:=e'_{q'+1}-e'_{p'}.$$
Then we obtain
$$V_{1,0}=f_0(Z_1)$$
and
$$\{V_{s,t}:(s,t)\in \mathbb{C}\times \mathbb{R}\}\subset S^{\ell'}_{q',p'}\cap Gr(q', N_\sharp^\perp).$$
Let $t\in \mathbb{R}$ be fixed. By Lemma~\ref{flat}, $f_0$ is linear on maximal complex submanifolds in $S^{\ell}_{q,p}$ and therefore we obtain
$$\{V_{s,t}:s\in \mathbb{C}\}=f_0\left(Gr\left(q, Z_0+\mathbb{C}(\3 e_{q+1}-\sqrt{-1}t \Check e_{q+1})\right)\right).$$
In particular, the curve $\{V_{1,t}:t\in \mathbb{R}\}$ is a submanifold of $f_0(S^{\ell}_{q,p})$ passing through $f_0(Z_1)$.
By differentiating $V_{1,t}$ with respect to $t$ at $t=0$, we obtain
$$\Check e_{q'+1}'\in f_0(Z_1)+V_{Z_1}+Y_{Z_1}''.$$
Since
$$\3 e'_{q'+1}\in f(Z_1)+V_{Z_1}$$
and
$$e'_{q'+1}=\frac12(\3 e'_{q'+1}+\Check e'_{q'+1}), \quad e'_{p'}=\frac12(\3e'_{q'+1}-\Check e'_{q'+1}),$$
this implies
$$e'_{q'+1},e'_{p'} \in f_0(Z_1)+V_{Z_1}+Y_{Z_1}''.$$
Since
$$\langle \3 e_1'+\3e'_{q'+1},e'_{q'+1}+\Check e'_1\rangle_{\ell',m'}
=\langle \3 e_1'+\3e'_{q'+1},e'_{p'}-\Check e'_1\rangle_{\ell',m'}=0$$
and
$$\langle \3 e_a',e'_{q'+1}+\Check e'_1\rangle_{\ell',m'}
=\langle \3 e_a',e'_{p'}-\Check e'_1\rangle_{\ell',m'}=0,~a=2,\ldots,q',$$
we obtain
$$e'_{q'+1}+\Check e'_1,e'_{p'}-\Check e'_1\in f(Z_1)+V_{Z_1},$$
i.e.
$$V_{Z_1}=\text{\rm span}\{e'_{q'+1}+\Check e'_1,e'_{p'}-\Check e'_1, e'_{q'+j},~1<j\leq \ell-q \text{  or  }(p'-q')-(m-q)< j<p'-q'\}\subset L.$$

Finally, we will show
$$Y_{Z_1}''=Y_{Z_0}''\mod f_0(Z_1)+V_{Z_1},$$
which will imply
$$f(Z_1)+V_{Z_1}+Y_{Z_1}''=L\oplus\text{\rm span}\{\3e'_a,~a=q+1,\ldots,q'\}$$
as desired.
Since
$$\langle e'_{q'+1}+\Check e'_1, \Check e'_\a\rangle_{\ell',m'}
=\langle e'_{p'}-\Check e'_1, \Check e_\a'\rangle_{\ell',m'}
=\langle e'_J,\Check e_\a'\rangle_{\ell',m'}=0,
~q'+1<J<p',$$
we obtain $Y_{Z_0}''=\text{\rm span}\{\Check e'_\a,~\a=1,\ldots,q\}$ is orthogonal to $X'_{Z_1}$ with respect to $\langle~,\rangle_{\ell',m'}.$ Since
$$\langle Z'_a,\Check e_b'\rangle_{\ell',m'}=\delta_{ab}$$
by direct computation,
$\{Z'_a, X_J' ,\Check e'_a\}$ is an $S^{\ell'}_{q',p'}$-frame at $f(Z_1)$.
Since
$$f(Z_0)+X''_{Z_0}+Y_{Z_0}''=f(Z_1)+X''_{Z_1}+Y_{Z_1}''$$
and $Y_{Z_0}''$ is orthogonal to $f_0(Z_1)+X'_{Z_1}$ with respect to $\langle~,\rangle'_{\ell',m'}$
we obtain
$$Y_{Z_1}''=Y_{Z_0}''\mod f(Z_1)+V_{Z_1}.$$
\end{proof}
\medskip



{\it Proof of Theorem~\ref{main}}~:
Define
$$f_{L}:=\pi_{L}\circ f.$$
Then $f_L:S^{\ell}_{q,p}\to \pi_L(S^{\ell'}_{q',p'})$ is a CR embedding.
Since $\pi_L(S^{\ell'}_{q',p'})$ is equivalent to $S^{\ell}_{q,p}$, by \cite{Ng}, $ f_L$ is an isomorphism between $Gr(q, p)$ and $Gr(q, L)$.
Hence $f$ decomposes into
$$(f_L\oplus \text{\rm span}\{\3 e_a', a=q+1,\ldots,q'\},f_N),$$
where $f_L$ is an isomorphism and $f_N$ is a holomorphic map into $Gr(q', f(Z_0)+N_\sharp)$ as desired.
\medskip

{\it Proof of Corollary~\ref{main-cor}}~:
Under the condition \eqref{main-ineq}, we obtain
$$\dim Gr(q, E)>\dim Gr(q', f(Z_0)+N_\sharp)$$
for $\ell$-dimensional space $E$. Since Grassmannian is of Picard number one and therefore holomorphic maps between them are either finite to one or constant(See \S11 of \cite{M89}), we conclude that $f_N$ restricted to each maximal complex subspace $Gr(q, E)$ is a constant map, which completes the proof.

\end{document}